\documentclass[12pt,a4paper,reqno,twoside]{amsart}

\usepackage[english]{babel}
\usepackage{stmaryrd}
\usepackage{dsfont}
\usepackage[symbol*,ragged]{footmisc}
\usepackage[colorlinks,linkcolor=red,anchorcolor=blue,citecolor=blue,urlcolor=blue]{hyperref}
\usepackage{color,xcolor}

\usepackage{geometry}
\usepackage{amssymb}
\usepackage{amsmath}
\usepackage{mathrsfs}
\usepackage{amsfonts}
\usepackage{epsfig}

\usepackage{amsthm}
\usepackage{amsxtra}
\usepackage{bbding}
\usepackage{epsfig}
\usepackage{graphicx}
\usepackage{latexsym}
\usepackage{mathbbol}
\usepackage{bbold}

\usepackage{pifont}
\usepackage{wasysym}
\usepackage{skull}
\usepackage{float}

\DeclareSymbolFontAlphabet{\mathbb}{AMSb}
\DeclareSymbolFontAlphabet{\mathbbol}{bbold}

\usepackage{amscd}
\usepackage[all]{xy}
\allowdisplaybreaks[4]
\usepackage{setspace}

\theoremstyle{plain}
\newtheorem{theorem}{\normalfont\scshape Theorem}[section]
\newtheorem{proposition}{\normalfont\scshape Proposition}[section]
\newtheorem{lemma}[proposition]{\normalfont\scshape Lemma}

\newtheorem*{corollary*}{\normalfont\scshape Corollary}

\theoremstyle{remark}
\newtheorem*{remark*}{\normalfont\scshape Remark}
\newtheorem*{notation}{\normalfont\scshape Notation}

\numberwithin{equation}{section}
\addtocounter{footnote}{1}

\renewcommand{\footnoterule}{
	\kern -3pt
	\hrule width 2.5in height 0.4pt
	\kern 3pt
}

\makeatletter
\@ifundefined{MakeUppercase}{}{}
\makeatother

\begin{document}
\title[On Piatetski--Shapiro primes from almost primes]
	  {On Piatetski--Shapiro primes from almost primes}
	
\author[Yuhua Zhao, Jinjiang Li, Linji Long, Min Zhang]
       {Yuhua Zhao \quad \& \quad Jinjiang Li \quad \& Linji Long \quad \& Min Zhang}
	
\address{[Yuhua Zhao] Department of Mathematics, China University of Mining and Technology,
		Beijing 100083, People's Republic of China}
	
\email{\textcolor{black}{yuhua.zhao.math@gmail.com}}

\address{[Jinjiang Li] (Corresponding author) Department of Mathematics, China University of Mining
        and Technology, Beijing 100083, People's Republic of China}
	
\email{\textcolor{black}{jinjiang.li.math@gmail.com}}

\address{[Linji Long] Department of Mathematics, China University of Mining and Technology,
		Beijing 100083, People's Republic of China}
	
\email{\textcolor{black}{linji.long.math@gmail.com}}

\address{[Min Zhang] School of Applied Science, Beijing Information Science and Technology University,
		  Beijing 100192, People's Republic of China  }
	
\email{\textcolor{black}{min.zhang.math@gmail.com}}
	
\date{}
	
\footnotetext[1]{Jinjiang Li is the corresponding author.    
\quad\,\,{\textbf{MR(2020) Subject Classification}}: 11N05, 11N36, 11L07, 11L20

}

\begin{abstract}
Denote by $\mathcal{P}_r$ an almost--prime with at most $r$ prime factors, counted according to multiplicity. In this manuscript, it is established that, for any fixed $0.98353<\gamma<1$, there exist infinitely many primes of the form $p=[n^{1/\gamma}]$, where $n$ is an almost--prime $\mathcal{P}_7$. This result constitutes an improvement upon the previous result of Baker, Banks, Guo and Yeager \cite{Baker-Banks-Guo-Yeager-2014}, who showed that there exist infinitely many primes $p$ such that $p=[n^{1/\gamma}]$ with $n\in\mathcal{P}_8$ for $\gamma$ near to one.
\end{abstract}
	
\maketitle

\section{Introduction and main result}
Suppose that $\gamma\in(\frac{1}{2},1)$ is a fixed real number. The Piatetski--Shapiro sequences is defined by
\begin{equation*}
\mathscr{N}_{\gamma}:=\big\{[ n^{1/\gamma}]:\,n\in \mathbb{N}^+\big\}.
\end{equation*}
Such sequences have been named in honor of Piatetski--Shapiro, who \cite{Piatetski-Shapiro-1953}, in 1953, proved that $\mathscr{N}_{\gamma}$ contains infinitely many primes provided that $\gamma\in(\frac{11}{12},1)$. The prime numbers of the form $p=[ n^{1/\gamma}]$ are called \textit{Piatetski--Shapiro primes of type $\gamma$}. To be specific, in this notation Piatetski--Shapiro \cite{Piatetski-Shapiro-1953} illustrated that the associated counting function
\begin{equation*}
\pi_\gamma(x):=\#\big\{\textrm{prime}\,\, p\leqslant x:\,p=[ n^{1/\gamma}]\,\,\textrm{for some}\,\,
n\in\mathbb{N}^+ \big\}
\end{equation*}
satisfies the asymptotic property
\begin{equation}\label{pi-gamma-PNT}
\pi_{\gamma}(x)=\frac{x^{\gamma}}{\log x}(1+o(1))
\end{equation}
as $x\to\infty$. Since then, subsequent work, i.e. $\mathscr{N}_{\gamma}$ contains infinitely many primes,
repeatedly extended the admissible range of $\gamma$ for above asymptotic formula (e.g., see the literatures \cite{Kolesnik-1967,Leitmann-1975,Leitmann-1980,Heath-Brown-1983,Kolesnik-1985,Liu-Rivat-1992, Rivat-1992, Rivat-Sargos-2001}). The hitherto best asymptotic result currently available in this direction is due to Rivat and Sargos \cite{Rivat-Sargos-2001}, who reached $\gamma\in(\frac{2426}{2817},1)$. Moreover, a lower bound of $\pi_\gamma(x)$ with the expected order of magnitude was obtained by Rivat and Wu \cite{Rivat-Wu-2001} for $\gamma\in(\frac{205}{243},1)$. These developments provide the analytic background for the inverse
problem considered herein.

Many authors investigated the arithmetic properties of Piatetski--Shapiro sequences, and it is natural to ask whether certain properties also hold when we constrain the inner variable $n$ to special subsequences. For an analogue of (\ref{pi-gamma-PNT}), one can consider the subset of $\mathscr{N}_{\gamma}$ and expect the counting function
\begin{equation*}
\Pi_\gamma(x):=\#\big\{\textrm{prime}\,\, p\leqslant x:\,p=[ q^{1/\gamma}]\,\,
\textrm{for some prime number $q$}\big\}
\end{equation*}
satisfies the asymptotic property
\begin{equation}\label{eq1_2}
\Pi_\gamma(x)=\frac{x^\gamma}{(\log x)^2}(1+o(1)) \qquad (x\to \infty).
\end{equation}
Even when $\gamma>1$, this problem is highly nontrivial. It asks for primes represented by a linear expression in prime variable, a problem closely connected with the twin prime conjecture. Balog \cite{Balog-1987}  proved (\ref{eq1_2}) for $\gamma>6/5$ in 1987.

Let $\mathcal{P}_r$ denote an almost--prime with at most $r$ prime factors, counted according to multiplicity. Since (\ref{eq1_2}) remains intractable for $\gamma < 1$, a natural approximation is to relax the primality condition to almost--primes. In 1987, Balog \cite{Balog-1987} studied this problem in two symmetric directions, showing that the lower bound of expected order of magnitude
\begin{equation*}
\sum_{\substack{\mathcal{P}_s \leqslant x \\ [\mathcal{P}_r^{1/\gamma}] = \mathcal{P}_s}} 1 \gg \frac{x^\gamma}{\log^2 x}
\end{equation*}
holds for $(r,s) \in \{(1, 9), (9, 1), (2, 5), (5, 2)\}$. Since then, considerable efforts have been devoted to improving the parameters $r$ and $s$.
	
The first direction focuses on Piatetski--Shapiro sequences evaluated at primes, i.e., $r=1$. In 2016, Banks,
Guo and Shparlinski \cite{Banks-Guo-Shparlinski-2016} improved Balog's result by showing that there exist infinitely many primes $p$ such that $[p^{1/\gamma}]=\mathcal{P}_8$ provided that $\gamma \in (0.950479, 1)$. Very recently, Xue, Li and Zhang \cite{Xue-Li-Zhang-2024} further enhanced this result to $[p^{1/\gamma}]=\mathcal{P}_7$ for $\gamma\in(0.989,1)$. To be specific, they showed that the following estimates
\begin{equation*}
\big|\{[p^{1/\gamma}]\leqslant x:p\textrm{\,\,is prime},\,[p^{1 / \gamma}]=\mathcal{P}_7\} \big| \gg \frac{x^\gamma}{\log ^2 x}
\end{equation*}	
holds for all sufficiently large $x$ provided that $\gamma \in(\gamma_7, 1)$ with $\gamma_7=0.989$.
The second direction focuses on exploring the inverse problem of capturing primes within Piatetski--Shapiro sequences evaluated at almost--primes, i.e., $s=1$. In 2014, Baker, Banks, Guo and Yeager \cite{Baker-Banks-Guo-Yeager-2014} improved the result $(r,s)=(9,1)$ of Balog \cite{Balog-1987} to $(r,s)=(8,1)$. More exactly, they showed that
\begin{equation*}
\big|\{n \leqslant x : n \in \mathcal{P}_8 \textrm{ and } [n^{1/\gamma}]
\textrm{ is prime}\}\big| \gg \frac{x}{(\log x)^2},
\end{equation*}
provided that $76/77 < \gamma < 1$.

Motivated by the recent result of Xue, Li and Zhang \cite{Xue-Li-Zhang-2024}  in the first direction, by refining the weighted sieve of Richert \cite{Richert-1969} and exponential sum estimates, we shall improve the result of Baker, Banks, Guo and Yeager \cite{Baker-Banks-Guo-Yeager-2014} and establish the following theorem.
\begin{theorem}\label{Theorem-1}
For any fixed $0.98353<\gamma<1$, there exist infinitely many primes of the form $p=[n^{1/\gamma}]$, where $n$
is an almost--prime $\mathcal{P}_7$.  More precisely, there holds
\begin{equation*}
\big|\{n\leqslant x:n\in\mathcal{P}_7\,\,\textrm{and}\,\,[n^{1/\gamma}]=p\,\,
\textrm{is a prime}\}\big|\gg\frac{x}{(\log x)^2},
\end{equation*}
provided that $0.98353<\gamma<1$, where the implied constant in the symbol $\gg$ depends only on $\gamma$.
\end{theorem}
\begin{remark*}
The key point of improving the number $r$ such that $[\mathcal{P}_r^{1/\gamma}]=p$ is to enlarge the level $\xi=\xi(\gamma)$, for $\gamma$ near to 1, of the mean value theorem of Bombieri--Vinigradov's type over Piatetski--Shapiro sequence. In order to use weighted sieve to deal with the problem, we treat the error term according to the idea of
Richert \cite{Richert-1969} and the method of Chen \cite{Chen-1973}.
\end{remark*}
\begin{notation}
Throughout this paper, the parameter $x$ is always sufficiently large, whereas $\varepsilon$ and $\eta$ denote small positive quantities whose values may change from line to line. The letters $p$, with or without subscripts,
are reserved for primes. We use $[x],\{x\}$ and $\|x\|$ for integral part of $x$, the fractional part of $x$ and the distance from $x$ to the nearest integer, respectively. The symbol $\mathcal{P}_r$ denotes an integer having at most $r$ prime factors, counted with multiplicity. The functions $\Lambda(n)$, $\mu(n)$ and $\Omega(n)$ have their customary meanings: von Mangoldt's function, M\"{o}bius' function, and the number of total prime factors of $n$, respectively. We write $\mathscr{L}=\log x$; $e(t)=\exp (2 \pi i t)$; $\psi(t)=t-[t]-\frac{1}{2}$.
The notation $n \sim X$ means that $n$ runs through a subinterval of $(X/2, X]$, whose endpoints are not necessarily the same in the different occurrences and may depend on the outer summation variables. $f(x) \ll g(x)$ means that $f(x)=O(g(x))$; $f(x) \asymp g(x)$ means that $f(x) \ll g(x) \ll f(x)$.
\end{notation}

\section{Preliminary Lemmas}
We recall the one--dimensional sieve framework needed later. Let $\mathscr{A}$ be a finite set of integers, and $\mathscr{P}$ an infinite set of primes with its complementary prime set denoted by $\overline{\mathscr{P}}$.
For given $z\geqslant2$, define $P(z)$ by the product below
\begin{equation*}
P(z)=\prod_{\substack{p<z\\ p\in\mathscr{P}}} p.
\end{equation*}
Define the sifting function as
\begin{equation*}
S(\mathscr{A},\mathscr{P},z)=\big|\big\{a\in\mathscr{A}:(a,P(z))=1\big\}\big|.
\end{equation*}
For every divisor $d$ subject to $d|P(z)$, let $\mathscr{A}_d$ consist of those elements of $\mathscr{A}$
divisible by $d$. We assume that its cardinality $|\mathscr{A}_d|$ admits the decomposition
\begin{equation}\label{sieve-condi-1}
|\mathscr{A}_d|=\frac{\omega(d)}{d}X+r_d,\qquad \mu(d) \neq 0,\qquad (d,\overline{\mathscr{P}})=1,
\end{equation}
where $\omega(d)$ is a multiplicative function subject to $0\leqslant\omega(p)<p$, the scale $X$ does not depend
on $d$, while $r_d$ is the error term whose average is small enough for $X$ to approximate the cardinality of $\mathscr{A}$. Also, we further impose the standard one dimension regularity condition on average over $p$ in $\mathscr{P}$ for $\omega(p)$, i.e.,
\begin{equation}\label{sieve-condi-2}
\sum_{\substack{z_1\leqslant p<z_2\\ p\in\mathscr{P}}}\bigg(1-\frac{\omega(p)}{p}\bigg)^{-1}
\leqslant\frac{\log z_2}{\log z_1}\bigg(1+\frac{\mathcal{K}}{\log z_1}\bigg),
\end{equation}
which holds for all $z_2>z_1\geqslant2$, where $\mathcal{K}$ is a constant satisfying $\mathcal{K}\geqslant1$. For details of (\ref{sieve-condi-1}) and (\ref{sieve-condi-2}), one can see the arguments (4.12)--(4.15) on page 28 of Halberstam and Richert \cite{Halberstam-Richert-book}, and the arguments on page 205 of Iwaniec \cite{Iwaniec-1981}.
	\begin{lemma}\label{upper-lower-sieve}
		Suppose that the conditions (\ref{sieve-condi-1}) and (\ref{sieve-condi-2}) hold. Then we have
		\begin{align}
			S(\mathscr{A},\mathscr{P},z) \geqslant & \,\, XV(z)\big(f(s)+O\big(\log^{-1/3}D\big)\big)-R_D, \label{lower-sieve}
					\\
			S(\mathscr{A},\mathscr{P},z) \leqslant & \,\, XV(z)\big(F(s)+O\big(\log^{-1/3}D\big)\big)+R_D, \label{upper-sieve}
		\end{align}
where
\begin{equation*}
R_D=\sum_{\substack{d<D\\ d|P(z)}}|r_d|,\qquad s=\frac{\log D}{\log z},
\end{equation*}
\begin{equation}\label{V(z)-def}
V(z)=\mathcal{C}(\omega)\frac{e^{-C_0}}{\log z}\bigg(1+O\bigg(\frac{1}{\log z}\bigg)\bigg),
\end{equation}
\begin{equation}\label{C(omega)-def}
\mathcal{C}(\omega)=\prod_p\bigg(1-\frac{\omega(p)}{p}\bigg)\bigg(1-\frac{1}{p}\bigg)^{-1},
\end{equation}
where $C_0$ denotes the Euler's constant, $f(s)$ and $F(s)$ denote the lower and upper linear sieve functions, which are determined by the following differential--difference equation
\begin{equation}
\begin{cases}\label{diff-eq}
	\displaystyle F(s)=\frac{2e^{C_0}}{s},\quad f(s)=0, \quad 0<s\leqslant2,
						\\
\displaystyle
\frac{\mathrm{d}}{\mathrm{d}s}(sF(s))=f(s-1),\quad \frac{\mathrm{d}}{\mathrm{d}s}(sf(s))=F(s-1),\quad s\geqslant2.
\end{cases}
\end{equation}
\end{lemma}
\begin{proof}
		For (\ref{lower-sieve}) and (\ref{upper-sieve}), one can refer to (6), (7), (8) on page 209 of Iwaniec \cite{Iwaniec-1981}, while (\ref{diff-eq}) can be referred to as a special case with $\varkappa=1$, $\beta=2$ in (9) of Iwaniec \cite{Iwaniec-1981}. Moreover, for (\ref{V(z)-def}) and (\ref{C(omega)-def}) one can see (2.4) and (2.5) of Chapter $5$ in Halberstam and Richert \cite{Halberstam-Richert-book}.
	\end{proof}

	\begin{lemma}\label{psi-expansion}
		For any $H>1$, one has
		\begin{equation}\label{psi-expan}
			\psi(t)=-\sum_{0<|h|\leqslant H}\frac{e(th)}{2\pi i h}+O(g(t,H)),
		\end{equation}
		where
		\begin{equation*}
			g(t,H):=\min\bigg(1,\frac{1}{H\|t\|}\bigg)=\sum_{h=-\infty}^{\infty} b_h e(th),
		\end{equation*}
		and
		\begin{equation*}
			b_h\ll\min\bigg(\frac{\log2H}{H},\frac{1}{|h|},\frac{H}{|h|^2}\bigg).
		\end{equation*}
	\end{lemma}
	\begin{proof}
		See the arguments on page 245 of Heath--Brown \cite{Heath-Brown-1983}.
	\end{proof}

	\begin{lemma}\label{expo-pair-gernal}
		Suppose that $f(x):[a,b]\to\mathbb{R}$ has continuous derivatives of arbitrary order on $[a,b]$, where $1\leqslant a<b\leqslant 2a$. Suppose further that
		\begin{equation*}
			\big|f^{(j)}(x)\big| \asymp \lambda_1 a^{1-j},
			\qquad j \geqslant 1,
			\qquad x \in [a,b].
		\end{equation*}
		Then for any exponential pair $(\kappa,\ell)$, we have
		\begin{equation*}
			\sum_{a<n\leqslant b}e(f(n))\ll \lambda_1^\kappa a^\ell+\lambda_1^{-1}.
		\end{equation*}
	\end{lemma}
	\begin{proof}
		See (3.3.4) of Graham and Kolesnik \cite{Graham-Kolesnik-book}.
	\end{proof}

	\begin{lemma}\label{latticepoints}
		For $1/2<\gamma<1$, $J\geqslant1$, $L\geqslant1$, $D\geqslant1$, $\Delta>0$, let $\mathscr{N}(\Delta)$ denote the number of solutions of
		the following inequality
		\begin{equation*}
			\bigg|\frac{h_1\ell_1^{\gamma}}{d_1} -\frac{h_2\ell_2^{\gamma}}{d_2}\bigg|<\Delta, \qquad
			h_1,h_2\sim J, \quad  \ell_1,\ell_2\sim L,\quad  d_1,d_2\sim D.
		\end{equation*}
		Then we have
		\begin{equation*}
			\mathscr{N}(\Delta)\ll(JD)^\varepsilon\big(JDL+\Delta D^3JL^{2-\gamma}\big).
		\end{equation*}
	\end{lemma}
	\begin{proof}
		Let $h_1d_2=u_1$, $h_2d_1=u_2$. Then one has
		\begin{align*}
			\mathscr{N}(\Delta)
			\ll & \,\, (JD)^\varepsilon\cdot\#\bigg\{(u_1,u_2,\ell_1,\ell_2):\frac{JD}{4}<u_i\leqslant JD,
			\frac{L}{2}<\ell_i\leqslant L,\Big|u_1\ell_1^{\gamma}-u_2\ell_2^{\gamma}\Big|<\Delta D^2\bigg\}
			\nonumber \\
			\ll & \,\, (JD)^\varepsilon\cdot\#\bigg\{(u_1,u_2,\ell_1,\ell_2):\frac{JD}{4}<u_i\leqslant JD,\quad
			\frac{L}{2}<\ell_i\leqslant L,
			\nonumber \\
			& \,\, \hspace{12em}
			\bigg|\frac{u_1}{u_2}-\bigg(\frac{\ell_2}{\ell_1}\bigg)^{\gamma}\bigg|
			<16\Delta DJ^{-1}L^{-\gamma}\bigg\}.
		\end{align*}
		By using Lemma 1 of Fouvry and Iwaniec \cite{Fouvry-Iwaniec-1989} with parameters
		\begin{equation*}
			(\alpha,\beta,M,N,\Delta)=\big(1,\gamma,JD,L,\Delta DJ^{-1}L^{-\gamma}\big),
		\end{equation*}
		one derives that
		\begin{equation*}
			\mathscr{N}(\Delta)\ll(JD)^\varepsilon\big(JDL+\Delta D^3JL^{2-\gamma}\big),
		\end{equation*}
		which completes the proof of Lemma \ref{latticepoints}.
	\end{proof}

\begin{lemma}\label{Robert-Sargos-lemma}
Assume that $H\geqslant1, M\geqslant1, N\geqslant1, X>0$, and $\alpha,\beta,\gamma$ are fixed real numbers
such that $\alpha(\alpha-1)\beta\gamma\neq0$. Let $S$ be exponential sum defined as
\begin{equation*}
S=\sum_{h\sim H}\sum_{n\sim N}\Bigg|\sum_{m\sim M}
e\bigg(X\frac{m^{\alpha}h^{\beta}n^{\gamma}}{M^{\alpha}H^{\beta}N^{\gamma}}\bigg)\Bigg|,
\end{equation*}
for which the following inequality holds
\begin{equation*}
S\ll_{\varepsilon}(HNM)^{1+\varepsilon}\bigg(\bigg(\frac{X}{HNM^2}\bigg)^{1/4}+\frac{1}{M^{1/2}}+\frac{1}{X}\bigg).
\end{equation*}
\end{lemma}
\begin{proof}
See Theorem 3 of Robert and Sargos \cite{Robert-Sargos-2006}.
\end{proof}

\section{Preliminaries of Transformation of the Problem}
Let
\begin{equation*}
\mathscr{A}=\big\{n: n\leqslant x,\,[n^{1/\gamma}]\textrm{ is prime} \big\}.
\end{equation*}
For any $d\leqslant D$, where $D$ is a fixed power of $x$ to be specified later, we must estimate accurately the cardinality of
\begin{equation*}
\mathscr{A}_d=\big\{n\in \mathscr{A}: d \mid n\big\}.
\end{equation*}
Since $md \in \mathscr{A}$ if and only if
$p \leqslant (md)^{1/\gamma} < p + 1$ and $md \leqslant x$,
the cardinality of $\mathscr{A}_d$ equals to the number of primes $p\leqslant x^{1/\gamma}$ in which the interval $[p^{\gamma}d^{-1},(p+1)^{\gamma}d^{-1})$ contains a natural number, and thus
\begin{align}\label{A_d-asymp}
		 \#\mathscr{A}_d
= & \,\, \#\big\{n:n\leqslant x,\,\textrm{prime}\,p=[n^{1/\gamma}],\,n\equiv0\!\!\!\!\pmod d \big\}
                    \nonumber \\
= & \,\, \#\bigg\{m:\frac{p^\gamma}{d}\leqslant m<\frac{(p+1)^\gamma}{d},\,\,p\leqslant x^{1/\gamma}\bigg\}
=  \sum_{p\leqslant x^{1/\gamma}}\bigg(\bigg[-\frac{p^\gamma}{d}\bigg]
         -\bigg[-\frac{(p+1)^\gamma}{d}\bigg]\bigg)
				 \nonumber \\
= & \,\, \frac{1}{d}\sum_{p\leqslant x^{1/\gamma}}\big((p+1)^\gamma-p^\gamma\big)
		 +\sum_{p\leqslant x^{1/\gamma}}\bigg(\psi\bigg(-\frac{(p+1)^{\gamma}}{d}\bigg)
		-\psi\bigg(-\frac{p^{\gamma}}{d}\bigg)\bigg)
				\nonumber \\
= & \,\, \frac{1}{d}\mathscr{X}+R_d,
\end{align}
where
\begin{gather*}
\mathscr{X}=\sum_{p\leqslant x^{1/\gamma}}\big((p+1)^\gamma-p^\gamma\big)=\frac{\gamma x}{\log x}(1+o(1)),
				   \nonumber \\
R_d=\sum_{p\leqslant x^{1/\gamma}}\bigg(\psi\bigg(-\frac{(p+1)^\gamma}{d}\bigg)-
\psi\bigg(-\frac{p^\gamma}{d}\bigg)\bigg).
\end{gather*}
In order to use upper bound sieve and lower bound sieve, i.e., Lemma \ref{upper-lower-sieve}, it is sufficient to establish that, for some $\xi=\xi(\gamma)$, there holds the following mean value theorem
\begin{equation*}
\sum_{d\leqslant x^\xi}\Bigg|\sum_{p\leqslant x^{1/\gamma}}
\bigg(\psi\bigg(-\frac{(p+1)^{\gamma}}{d}\bigg)-\psi\bigg(-\frac{p^{\gamma}}{d}\bigg)\bigg)\Bigg|
\ll x\mathscr{L}^{-A},
\end{equation*}
which is equivalent to
\begin{equation}\label{mean-value-equilavent}
\sum_{d\leqslant x^\xi}\Bigg|\sum_{n\leqslant x^{1/\gamma}}\Lambda(n)
\bigg(\psi\bigg(-\frac{(n+1)^{\gamma}}{d}\bigg)-\psi\bigg(-\frac{n^{\gamma}}{d}\bigg)\bigg)\Bigg|
\ll x\mathscr{L}^{-A}.
\end{equation}
Set $D=x^\xi$. Trivially, a dyadic decomposition shows that (\ref{mean-value-equilavent}) is a consequence of the following uniform assertion for any $X\leqslant x^{1/\gamma}$, i.e.,
\begin{equation}\label{mean-value-split}
\sum_{d\leqslant D}\Bigg|\sum_{n\sim X}\Lambda(n)
\bigg(\psi\bigg(-\frac{(n+1)^{\gamma}}{d}\bigg)-\psi\bigg(-\frac{n^{\gamma}}{d}\bigg)\bigg)\Bigg|
\ll x\mathscr{L}^{-A}.
\end{equation}
Suppose that $\eta>0$ is sufficiently small. In the shorter range $X\leqslant x^{(1-\eta)/\gamma}$, the left--hand side of (\ref{mean-value-split}) admits the elementary bound
\begin{align}\label{1-trivial}
\ll & \,\, \sum_{d\leqslant D}\Bigg|\sum_{n\sim X}\Lambda(n)\frac{n^\gamma-(n+1)^\gamma}{d}\Bigg|
           +\sum_{d\leqslant D}\Bigg|\sum_{n\sim X}\Lambda(n)\bigg(\bigg[-\frac{n^{\gamma}}{d}\bigg]-
           \bigg[-\frac{(n+1)^\gamma}{d}\bigg]\bigg)\Bigg|
			      \nonumber \\
\ll & \,\, x^{1-\eta}+\mathscr{L}\sum_{d\leqslant D}\sum_{n\sim X}
           \bigg(\bigg[-\frac{n^{\gamma}}{d}\bigg]-\bigg[-\frac{(n+1)^\gamma}{d}\bigg]\bigg).
\end{align}
Applying Lemma \ref{psi-expansion} with $H=H_1:=x^{\xi+1/\gamma-1+\eta}$ in (\ref{psi-expan}) to the remaining sum in (\ref{1-trivial}), this yields
\begin{align*}
 & \,\, \sum_{d\leqslant D}\sum_{n\sim X}\bigg(\bigg[-\frac{n^{\gamma}}{d}\bigg]-
        \bigg[-\frac{(n+1)^\gamma}{d}\bigg]\bigg)
		       \nonumber \\
= & \,\, \sum_{d\leqslant D}\sum_{n\sim X}\bigg(\frac{(n+1)^\gamma-n^\gamma}{d}+
		 \psi\bigg(-\frac{(n+1)^\gamma}{d}\bigg)-\psi\bigg(-\frac{n^{\gamma}}{d}\bigg)\bigg)
		       \nonumber \\
= & \,\, \sum_{d\leqslant D}\sum_{n\sim X}\Bigg(\frac{(n+1)^\gamma-n^\gamma}{d}
         +\sum_{0<|h|\leqslant H_1}\frac{1}{2\pi ih}
		 \bigg(e\bigg(\frac{hn^{\gamma}}{d}\bigg)-e\bigg(\frac{h(n+1)^{\gamma}}{d}\bigg)\bigg)
		       \nonumber \\
& \,\, \hspace{4.5em} + O\bigg(g\bigg(-\frac{n^{\gamma}}{d},H_1\bigg)\bigg)
       +O\bigg(g\bigg(-\frac{(n+1)^{\gamma}}{d},H_1\bigg)\bigg)\Bigg)
		       \nonumber\\
=: & \,\, S_0+S_1+S_2+S_3.
\end{align*}
Trivially, $S_0\ll X^\gamma\mathscr{L}\ll x^{1-\eta}$. For $S_2$ and $S_3$, it follows from Lemma \ref{expo-pair-gernal} with $(\kappa, \ell)=(\frac{1}{2}, \frac{1}{2})$ that
\begin{align*}
		   S_2,S_3
\ll & \,\, \sum_{d\leqslant D}\sum_{h=-\infty}^\infty |b_h|\bigg|\sum_{n\sim X}
           e\bigg(\frac{hn^{\gamma}}{d}\bigg)\bigg|
		          \nonumber \\
\ll & \,\, |b_0|DX+\sum_{d\leqslant D}\sum_{\substack{h=-\infty\\ h\not=0}}^\infty|b_h|
		   \Big(d|h|^{-1}X^{1-\gamma}+d^{-1/2}|h|^{1/2}X^{\gamma/2}\Big)
		          \nonumber \\
\ll & \,\, \mathscr{L}DXH_1^{-1}+\sum_{d\leqslant D}\sum_{0<|h|\leqslant H_1}|h|^{-1}
		   \Big(d|h|^{-1}X^{1-\gamma}+d^{-1/2}|h|^{1/2}X^{\gamma/2}\Big)
		          \nonumber \\
    & \,\, +\sum_{d\leqslant D}\sum_{|h|>H_1}H_1|h|^{-2}
		   \Big(d|h|^{-1}X^{1-\gamma}+d^{-1/2}|h|^{1/2}X^{\gamma/2}\Big)
		          \nonumber \\
\ll & \,\, \mathscr{L}DXH_1^{-1}+D^2X^{1-\gamma}+D^{1/2}H_1^{1/2}X^{\gamma/2}
		          \nonumber \\
\ll & \,\, x^{1-\eta}+x^{2\xi}\cdot x^{1/\gamma-1}+x^{\xi+1/(2\gamma)}
		          \nonumber \\
\ll & \,\, x\mathscr{L}^{-A},
\end{align*}
provided that
\begin{equation*}
\xi + \frac{1}{2\gamma}<1.
\end{equation*}
For $S_1$, by Lemma \ref{expo-pair-gernal} with $(\kappa,\ell)=(\frac{1}{2},\frac{1}{2})$, we have
\begin{align*}
		 S_1
= & \,\, \sum_{d\leqslant D}\sum_{0<|h|\leqslant H_1}\frac{1}{2\pi ih}\sum_{n\sim X}
		 \bigg(e\bigg(\frac{hn^{\gamma}}{d}\bigg)-e\bigg(\frac{h(n+1)^{\gamma}}{d}\bigg)\bigg)
		        \nonumber \\
= & \,\, \sum_{d\leqslant D}\sum_{0<|h|\leqslant H_1}\frac{1}{2\pi ih}\sum_{n\sim X}(-2\pi i\gamma)\frac{h}{d}
		 \int_{0}^{1}(n+u)^{\gamma-1}e\bigg(\frac{h(n+u)^{\gamma}}{d}\bigg) \mathrm{d}u
		        \nonumber \\
\ll & \,\, \sum_{d\leqslant D}\frac{1}{d}\sum_{0<|h|\leqslant H_1}\max_{0\leqslant u\leqslant 1}
		   \Bigg|\sum_{n\sim X}(n+u)^{\gamma-1}e\bigg(\frac{h(n+u)^{\gamma}}{d}\bigg)\Bigg|
		        \nonumber \\
\ll & \,\, \sum_{d\leqslant D}\frac{X^{\gamma-1}}{d}\sum_{0<|h|\leqslant H_1}\max_{0\leqslant u\leqslant 1}
		   \Bigg|\sum_{n\sim X}e\bigg(\frac{h(n+u)^{\gamma}}{d}\bigg)\Bigg|
		        \nonumber \\
\ll & \,\, \sum_{d\leqslant D}\frac{X^{\gamma-1}}{d}\sum_{0<|h|\leqslant H_1}	
           \Big(d|h|^{-1}X^{1-\gamma}+d^{-1/2}|h|^{1/2}X^{\gamma/2}\Big)
		        \nonumber \\
\ll & \,\, x^{\xi+\eta}+x^{3\xi/2+1/(2\gamma)+\eta}
		        \nonumber \\
\ll & \,\, x\mathscr{L}^{-A},
\end{align*}
provided that
\begin{equation}\label{suffi-condi-1}
\frac{3}{2}\xi+\frac{1}{2\gamma}<1.
\end{equation}
Therefore, it suffices to show that (\ref{mean-value-split}) still holds for
$x^{(1-\eta)/\gamma}<X\leqslant x^{1/\gamma}$. It is easy to see that, for $\xi \leqslant (1-\eta)/(2\gamma)$, there holds
\begin{equation*}
X^{\gamma\xi}\leqslant D\leqslant X^{\gamma\xi+\eta/2}.
\end{equation*}
	Putting (\ref{psi-expan}) into the left--hand side of (\ref{mean-value-split}) with $H=H_2:=x^{\xi+1/\gamma-1+\eta}$, the contribution of the error term in (\ref{psi-expan}) to the left--hand side of (\ref{mean-value-split}) is
\begin{equation*}
\sum_{d\leqslant D}\sum_{n\sim X}\Lambda(n)\bigg(g\bigg(\frac{n^{\gamma}}{d},H_2\bigg)
+g\bigg(\frac{(n+1)^{\gamma}}{d},H_2\bigg)\bigg)=:E_1+E_2,
\end{equation*}
say. For $E_1$ and $E_2$, uniformly, by Lemma \ref{expo-pair-gernal} with $(\kappa,\ell)=(\frac{1}{2},\frac{1}{2})$ and partial summation, we have
\begin{align*}
		   E_1,E_2
\ll & \,\, \mathscr{L}\sum_{d\leqslant D}\sum_{h=-\infty}^\infty|b_h|
		   \Bigg|\sum_{n\sim X}e\bigg(\frac{hn^{\gamma}}{d}\bigg)\Bigg|
		         \nonumber \\
\ll & \,\, \mathscr{L}\sum_{d\leqslant D}\Bigg(|b_0|X+\sum_{\substack{h=-\infty\\ h\not=0}}^\infty
		   |b_h|\bigg(\frac{d}{|h|}X^{1-\gamma}+\bigg(\frac{|h|}{d}\bigg)^{1/2}X^{\gamma/2}\bigg)\Bigg)
		         \nonumber \\
\ll & \,\, \mathscr{L}^2DXH_2^{-1}+\mathscr{L}X^{1-\gamma}\sum_{d\leqslant D}d
		   \sum_{\substack{h=-\infty\\ h\not=0}}^\infty\frac{1}{h^2}
		         \nonumber \\
	& \,\, +\mathscr{L}X^{\gamma/2}\sum_{d\leqslant D}d^{-1/2}\Bigg(\sum_{0<|h|\leqslant H_2}|h|^{-1/2}+
		   \sum_{|h|>H_2}\frac{H_2}{|h|^{3/2}}\Bigg)
		         \nonumber \\
\ll & \,\, \mathscr{L}^2DXH_2^{-1}+\mathscr{L}D^2X^{1-\gamma}+\mathscr{L}X^{\gamma /2}D^{1/2}H_2^{1/2}
		         \nonumber \\
\ll & \,\, x\mathscr{L}^{-A},
\end{align*}
provided that
\begin{equation}\label{suffi-condi-2}
\xi+\frac{1}{2\gamma}<1.
\end{equation}
Thus, it remains to show that
\begin{equation}\label{L-X-Main}
\mathcal{S}:=\sum_{d\leqslant D}\Bigg|\sum_{0<|h|\leqslant H_2}\frac{1}{h}\sum_{n\sim X}\Lambda(n)
\bigg(e\bigg(\frac{hn^{\gamma}}{d}\bigg)-e\bigg(\frac{h(n+1)^{\gamma}}{d}\bigg)\bigg)\Bigg|\ll x\mathscr{L}^{-A}.
\end{equation}
For the left--hand side of (\ref{L-X-Main}), one has
\begin{align*}
		 \mathcal{S}
= & \,\, \sum_{d\leqslant D}\Bigg|\sum_{0<|h|\leqslant H_2}\frac{1}{h}\sum_{n\sim X}
         \Lambda(n)(-2\pi i\gamma)\frac{h}{d}\int_{0}^{1}(n+u)^{\gamma-1}
         e\bigg(\frac{h(n+u)^{\gamma}}{d}\bigg)\mathrm{d}u\Bigg|
		        \nonumber \\
\ll & \,\, \sum_{d\leqslant D}\frac{1}{d}\sum_{0<h\leqslant H_2}\Bigg|\int_{0}^{1}\sum_{n\sim X}\Lambda(n)
		   (n+u)^{\gamma-1}e\bigg(\frac{h(n+u)^{\gamma}}{d}\bigg)\mathrm{d}u\Bigg|
		        \nonumber \\
\ll & \,\, \sum_{d\leqslant D}\frac{1}{d}\sum_{0<h\leqslant H_2}\max_{0\leqslant u\leqslant 1}\Bigg|
		   \sum_{n\sim X}\Lambda(n)(n+u)^{\gamma-1}e\bigg(\frac{h(n+u)^{\gamma}}{d}\bigg)\Bigg|
		        \nonumber \\
\ll & \,\, X^{\gamma-1}\sum_{d\leqslant D}\frac{1}{d}\sum_{0<h\leqslant H_2}
           \max_{0\leqslant u\leqslant 1}\Bigg|\sum_{n\sim X}\Lambda(n)
           e\bigg(\frac{h(n+u)^{\gamma}}{d}\bigg)\Bigg|.
\end{align*}
It is sufficient to show that, uniformly for $u\in[0,1]$, there holds
\begin{align*}
X^{\gamma-1}\sum_{d\leqslant D}\sum_{0<h\leqslant H_2}\frac{\delta(d,h)}{d}\sum_{n\sim X}\Lambda(n)
e\bigg(\frac{h(n+u)^{\gamma}}{d}\bigg)\ll x\mathscr{L}^{-A},
\end{align*}
which implies that
\begin{align*}
\sum_{n\sim X}\Lambda(n)G(n):=\sum_{n\sim X}\Lambda(n)\sum_{d\leqslant D}\sum_{0<h\leqslant H_2}
\frac{\delta(d,h)}{d}e\bigg(\frac{h(n+u)^{\gamma}}{d}\bigg)\ll X\mathscr{L}^{-A},
\end{align*}
where
\begin{equation*}
G(n)=\sum_{d\leqslant D}\sum_{0<h\leqslant H_2}\frac{\delta(d,h)}{d}e\bigg(\frac{h(n+u)^{\gamma}}{d}\bigg),
\qquad|\delta(d,h)|=1.
\end{equation*}
A special case of the identity of Heath--Brown \cite{Heath-Brown-1982} is given by
\begin{equation*}
-\frac{\zeta'}{\zeta}=-\frac{\zeta'}{\zeta}(1-Z\zeta)^3-\sum_{j=1}^3\binom{3}{j}(-1)^jZ^j\zeta^{j-1}(-\zeta'),
\end{equation*}
where $Z=Z(s)=\sum\limits_{m\leqslant X^{1/3}}\mu(m)m^{-s}$. From this we can decompose $\Lambda(n)$
for $n\sim X$ as
\begin{equation*}
\Lambda(n)=\sum_{j=1}^3\binom{3}{j}(-1)^{j-1}\sum_{m_1\dots m_{2j}=n}\mu(m_1)\dots\mu(m_j)\log m_{2j}.
\end{equation*}
Thus, for any arithmetic function $G(n)$, we can express $\sum\limits_{n\sim X}\Lambda(n)G(n)$ in terms of sums
\begin{equation*}
\mathop{\sum\,\,\dots\,\,\sum}_{\substack{m_1\dots m_{2j}\sim X\\ m_i\sim M_i}}\mu(m_1)\dots\mu(m_j)
(\log m_{2j})G(m_1\dots m_{2j}),
\end{equation*}
where $1\leqslant j\leqslant3$, $M_1M_2\dots M_{2j}\sim X$ and $M_1,\dots,M_j\leqslant X^{1/3}$. By dividing
the $M_j$ into two groups, we have
\begin{equation}\label{expo-fenjie}
\Bigg|\sum_{n\sim X}\Lambda(n)G(n)\Bigg|\ll_\eta X^\eta\max\Bigg|
\mathop{\sum\sum}_{\substack{m\ell\sim X\\ m\sim M}}a(m)b(\ell)G(mn)\Bigg|,
\end{equation}
where the maximum is taken over all bilinear forms with coefficients satisfying one of
\begin{equation}\label{type-II-coeff}
|a(m)|\leqslant 1,\qquad \qquad |b(\ell)|\leqslant1,
\end{equation}
or
\begin{equation*}
|a(m)|\leqslant 1,\qquad \qquad b(\ell)=1,
\end{equation*}
or
\begin{equation*}
|a(m)|\leqslant 1,\qquad \qquad b(\ell)=\log \ell,
\end{equation*}
and also satisfying in all cases
\begin{equation}\label{gene-coeff-condi}
M\leqslant X.
\end{equation}
We refer to the case (\ref{type-II-coeff}) as being Type II sums and to the other cases as being Type I sums and write for brevity $\Sigma_{II}$ and $\Sigma_{I}$, respectively. By dividing the $M_j$ into two groups in a judicious fashion we are able to reduce the range of $M$ from (\ref{gene-coeff-condi}). In Section \ref{ex-section}, we shall give the estimate of these sums.

\section{Estimate of Exponential Sums}\label{ex-section}
This section supplies the exponential sum bounds, which are required for the proof of Theorem \ref{Theorem-1}. The estimates are arranged according to the Type II and Type I ranges arising from the preceding decomposition.
\subsection{The Estimate of Type II Sums}\label{subse-type-II}
We commence by breaking up the ranges for $\ell$ and $h$ into intervals $(L/2,L]$ and $(J/2,J]$ so
that $ML\asymp X$ and $1\ll J\leqslant H_2$. After this decomposition, the contribution of Type II sums $\Sigma_{II}$ satisfies
\begin{equation*}
\Sigma_{II}\ll\mathscr{L}^3\sum_{m\sim M}\Bigg|\sum_{\substack{\ell\sim L\\ m\ell\sim X}}b(\ell)
\sum_{h\sim J}\sum_{d\sim D}\frac{\delta(d,h)}{d}e\bigg(\frac{h(m\ell+u)^{\gamma}}{d}\bigg)\Bigg|.
\end{equation*}
Then we have
\begin{equation*}
0<\frac{h(\ell+u/m)^{\gamma}}{d}\leqslant\frac{2J(L+1)^{\gamma}}{D}.
\end{equation*}
Denote by $T$ an auxiliary parameter, which will be optimized below. For any $u\in[0,1]$, we decompose the collection of the admissible triples $(h,\ell,d)$ into classes $\mathscr{S}_{t,u}\,(1\leqslant t\leqslant T)$, defined by
\begin{equation*}
\mathscr{S}_{t,u}=\bigg\{(h,\ell,d):\,h\sim J, \ell\sim L, d\sim D,
\frac{2J(L+1)^{\gamma}(t-1)}{DT}<\frac{h(\ell+u/m)^{\gamma}}{d}\leqslant\frac{2J(L+1)^{\gamma}t}{DT}\bigg\}.
\end{equation*}
Therefore, we have
\begin{equation*}
\Sigma_{II}\ll \mathscr{L}^3\sum_{1\leqslant t\leqslant T}\sum_{m\sim M}
\Bigg|\mathop{\sum\sum\sum}_{\substack{(h,\ell,d)\in\mathscr{S}_{t,u}\\ m\ell\sim X}}
b(\ell)\frac{\delta(d,h)}{d}e\bigg(\frac{h(m\ell+u)^{\gamma}}{d}\bigg)\Bigg|,
\end{equation*}
which combined with Cauchy's inequality yields that
\begin{align}\label{Sigma-2-upper-1}
		   |\Sigma_{II}|^2
\ll & \,\, \mathscr{L}^6TM\sum_{1\leqslant t\leqslant T}\sum_{m\sim M}
		   \Bigg|\mathop{\sum\sum\sum}_{\substack{(h,\ell,d)\in\mathscr{S}_{t,u}\\ m\ell\sim X}}
		   b(\ell)\frac{\delta(d,h)}{d}e\bigg(\frac{h(m\ell+u)^{\gamma}}{d}\bigg)\Bigg|^2
		          \nonumber \\
\ll & \,\, \mathscr{L}^6TM\sum_{1\leqslant t\leqslant T}\mathop{\sum\sum\sum}_{(h_1,\ell_1,d_1)\in
           \mathscr{S}_{t,u}}\mathop{\sum\sum\sum}_{(h_2,\ell_2,d_2)\in\mathscr{S}_{t,u}}\frac{1}{d_1d_2}
		          \nonumber\\
    & \,\, \times\Bigg|\sum_{\substack{m\sim M\\ m\ell_1\sim X \\m\ell_2\sim X}}
			e\bigg(\frac{h_1(m\ell_1+u)^{\gamma}}{d_1}-\frac{h_2(m\ell_2+u)^{\gamma}}{d_2}\bigg) \Bigg|
		          \nonumber \\
\ll & \,\, \mathscr{L}^6TMD^{-2}\mathop{\sum_{h_1\sim J}\sum_{h_2\sim J}\sum_{\ell_1\sim L}\sum_{\ell_2\sim L}
		   \sum_{d_1\sim D}\sum_{d_2\sim D}}_{|\lambda|\leqslant2J(L+1)^{\gamma}D^{-1}T^{-1}}\Bigg|
		   \sum_{\substack{m\sim M\\ m\ell_1\sim X \\m\ell_2\sim X}}e(\mathfrak{g}_u(m))\Bigg|,
\end{align}
where the function $\mathfrak{g}_u(m)$ is defined as
\begin{equation*}
\mathfrak{g}_u(m)=\frac{h_1(m\ell_1+u)^\gamma}{d_1}-\frac{h_2(m\ell_2+u)^\gamma}{d_2}.
\end{equation*}
Define
\begin{equation*}
\lambda=\frac{h_1\ell_1^\gamma}{d_1}-\frac{h_2\ell_2^\gamma}{d_2}.
\end{equation*}
For $i\geqslant 1$, set $\mathfrak{C}(\gamma,i)=\gamma(\gamma-1)\dots(\gamma-i+1)$. Then, for $j\geqslant1$,
the $j$--th derivative of $\mathfrak{g}_u(m)$ with respect to $m$ satisfies
\begin{align}\label{g-j-derivative}
          \mathfrak{g}_u^{(j)}(m)
 = & \,\, \mathfrak{C}(\gamma,j)\bigg(\frac{h_1\ell_1^j(m\ell_1+u)^{\gamma-j}}{d_1}-
          \frac{h_2\ell_2^j(m\ell_2+u)^{\gamma-j}}{d_2} \bigg)
                 \nonumber \\
 = & \,\, \mathfrak{C}(\gamma,j)m^{\gamma-j}\bigg(\frac{h_1\ell_1^\gamma(1+u/(m\ell_1))^{\gamma-j}}{d_1}-
          \frac{h_2\ell_2^\gamma(1+u/(m\ell_2))^{\gamma-j}}{d_2} \bigg)
                 \nonumber \\
 = & \,\, \mathfrak{C}(\gamma,j)m^{\gamma-j}\bigg(\frac{h_1\ell_1^\gamma}{d_1}-\frac{h_2\ell_2^\gamma}{d_2}
          +O\bigg(\frac{JL^\gamma}{DX}\bigg)\bigg)
                 \nonumber \\
 = & \,\, \mathfrak{C}(\gamma,j)m^{\gamma-j}\bigg(\lambda+O\bigg(\frac{JL^\gamma}{DX}\bigg)\bigg),
\end{align}
say. By noting that $1\ll J\leqslant H_2$, we deduce that $JL^\gamma D^{-1}X^{-1}\ll M^{-\gamma+\eta}$. If
$|\lambda|\leqslant M^{-\gamma+\eta}$, we use the trivial estimate $M$ to bound the innermost sum in the absolute
value in (\ref{Sigma-2-upper-1}). By Lemma \ref{latticepoints}, the total contribution of the trivial upper bound estimate $M$ to $|\Sigma_{II}|^2$ is
\begin{align}\label{lambda-s-upp}
\ll & \,\, \mathscr{L}^6TM^2D^{-2}\mathop{\sum_{h_1\sim J}\sum_{h_2\sim J}\sum_{\ell_1\sim L}
		   \sum_{\ell_2\sim L}\sum_{d_1\sim D}\sum_{d_2\sim D}}_{|\lambda|\leqslant M^{-\gamma+\eta}}1
		           \nonumber \\
\ll & \,\, X^\eta TM^2D^{-2}\big(JDL+M^{-\gamma}JD^3L^{2-\gamma}\big)
		           \nonumber \\
\ll & \,\, X^\eta \big(XMTJD^{-1}+X^{2-\gamma}TJD\big).
\end{align}	
If $|\lambda|> M^{-\gamma+\eta}$, it follows from (\ref{g-j-derivative}) that
\begin{equation*}
\big| \mathfrak{g}_u^{(j)}(m)\big|\asymp|\lambda|M^{\gamma-1}\cdot M^{1-j},
\end{equation*}
which combined with Lemma \ref{expo-pair-gernal} with exponential pair $(\kappa,\ell)=BA^2BABA^2B(0,1)=(\frac{11}{32},\frac{13}{24})$ yields that
\begin{align*}
|\Sigma_{II}|^2\ll \mathscr{L}^6TMD^{-2}\mathop{\sum_{h_1\sim J}\sum_{h_2\sim J}\sum_{\ell_1\sim L}
\sum_{\ell_2\sim L}\sum_{d_1\sim D}\sum_{d_2\sim D}}_{|\lambda|\leqslant2J(L+1)^{\gamma}D^{-1}T^{-1}}
\min\Big\{M,|\lambda|^{-1}M^{1-\gamma}+|\lambda|^{11/32}M^{11\gamma/32+19/96}\Big\}.
\end{align*}
It follows from the splitting argument that the contribution of the $|\lambda|^{-1}M^{1-\gamma}$ term to $|\Sigma_{II}|^2$ is
\begin{align}\label{lambda-l-upp}
\ll & \,\, \mathscr{L}^7TM^{2-\gamma}D^{-2}
		   \max_{M^{-\gamma+\eta}\leqslant\Delta\leqslant2J(L+1)^{\gamma}D^{-1}T^{-1}}
           \mathscr{N}(\Delta)\Delta^{-1}
		             \nonumber \\
\ll & \,\, X^\eta TM^{2-\gamma}D^{-2}\big(JDLM^{\gamma}+JD^3L^{2-\gamma}\big)
		             \nonumber \\
\ll & \,\, X^\eta\big(XMTJD^{-1}+X^{2-\gamma}TJD\big).
\end{align}
Moreover, the total contribution of the term $|\lambda|^{11/32}M^{11\gamma/32+19/96}$ to $|\Sigma_{II}|^2$ is
\begin{align*}
		\ll & \,\, |\lambda|^{11/32}M^{11\gamma/32+19/96}\cdot\mathscr{L}^6TMD^{-2}
		\cdot\mathscr{N}(2J(L+1)^{\gamma}D^{-1}T^{-1})
		\nonumber \\
		\ll & \,\, X^\eta\big(JL^{\gamma}D^{-1}T^{-1}\big)^{11/32}TM^{11\gamma/32+115/96}D^{-2}
		\big(JDL+JL^{\gamma}D^{-1}T^{-1}\cdot JD^3L^{2-\gamma}\big)
		\nonumber \\
		\ll & \,\, X^\eta\big(
		X^{11\gamma/32+1}M^{19/96}J^{43/32}D^{-43/32}T^{21/32}
		+X^{11\gamma/32+2}M^{-77/96}J^{75/32}D^{-11/32}T^{-11/32}
		\big),
\end{align*}
which combined with (\ref{lambda-s-upp}) and (\ref{lambda-l-upp}) yields that
\begin{align}\label{Sigma-2-fi-1}
		|\Sigma_{II}|^{2}
		\ll & \,\, X^\eta\big(XMTJD^{-1}+X^{2-\gamma}TJD
		+X^{11\gamma/32+1}M^{19/96}J^{43/32}D^{-43/32}T^{21/32}
		\nonumber \\
		& \,\, \qquad + X^{11\gamma/32+2}M^{-77/96}J^{75/32}D^{-11/32}T^{-11/32}\big).
\end{align}
Now, we choose
\begin{equation}\label{T-chosen}
T=\big[X^{(11\gamma+32)/43}M^{-173/129}D^{64/43}\big].
\end{equation}
Trivially, one has $T>1$ provided that $M\ll X^{(33\gamma+192\gamma\xi+96)/173-\eta}$. Inserting (\ref{T-chosen}) into (\ref{Sigma-2-fi-1}),  we obtain
\begin{align*}
		|\Sigma_{II}|^2
		\ll & \,\, X^\eta\Big(
		M^{-44/129}X^{(11\gamma+75)/43}JD^{21/43}
		+M^{-173/129}X^{(118-32\gamma)/43}JD^{107/43}
		\nonumber \\
		& \,\,  +M^{-88/129}X^{(22\gamma+64)/43}J^{43/32}D^{-505/1376}
		+M^{-44/129}X^{(11\gamma+75)/43}J^{75/32}D^{-1177/1376}\Big),
\end{align*}
which combined with $J\ll H_2=x^{\xi+1/\gamma-1+\eta}$
=$X^{\xi\gamma+1-\gamma+\eta}$
and $X^{\xi\gamma}\leqslant D\leqslant X^{\xi\gamma+\eta/2}$ yields that
\begin{align*}
		|\Sigma_{II}|^2
		\ll & \,\, X^\eta\Big(
		M^{-44/129}X^{(5625-2873\gamma)/1376+64\gamma\xi/43}
		+M^{-173/129}X^{(161-75\gamma+150\gamma\xi)/43}
		\nonumber \\
		& \,\, \qquad +M^{-88/129}X^{(3897-1145\gamma)/1376+42\gamma\xi/43}\Big).
\end{align*}
According to above arguments, we deduce the following lemma.
\begin{lemma}\label{Type-II-es}
Suppose that $\frac{1}{2}<\gamma<1$ and $0<\xi\leqslant(1-\eta)/(2\gamma)$ satisfy the condition
\begin{equation*}
\frac{3}{2}\xi+\frac{1}{2\gamma}<1.
\end{equation*}
If there holds
\begin{equation*}
X^{8619(1-\gamma)/1408+48\gamma\xi/11+\eta}\ll M\ll X^{(33\gamma+192\gamma\xi+96)/173-\eta},
\end{equation*}
then we have
\begin{equation*}
\Sigma_{II}\ll X^{1-\eta}.
\end{equation*}
\end{lemma}

\subsection{The Estimate of Type I Sums}
As in Subsection \ref{subse-type-II}, we begin by breaking up the range for $\ell$ into intervals $(L/2,L]$,
such that $ML\asymp X$. Then we change the order of summation and derive that
\begin{equation*}
\Sigma_I\ll\sum_{d\leqslant D}\frac{1}{d}\sum_{0<|h|\leqslant H_2}\sum_{m\sim M}\sum_{\ell\sim L}a(m)b(\ell)
\delta(d,h)e\bigg(\frac{h(m\ell+u)^{\gamma}}{d}\bigg),
\end{equation*}
where $u\in[0,1]$. The dyadic splitting reduces the argument to proving, for each fixed $d\sim D$, that
\begin{equation}\label{Type-I-inner}
\sum_{h\sim H_2}\sum_{m\sim M}\sum_{\ell\sim L}a(m)b(\ell)\delta(d,h)
e\bigg(\frac{h(m\ell+u)^{\gamma}}{d}\bigg)\ll X^{1-\eta}.
\end{equation}
Denote by $\mathcal{K}_d$ the left--hand side of (\ref{Type-I-inner}). By partial summation, one has
\begin{align*}
        \mathcal{K}_d
= &\,\, \sum_{h\sim H_2}\sum_{m\sim M}a(m)\delta(d,h)\sum_{\ell\sim L}b(\ell)
        e\bigg(\frac{h(m\ell+u)^{\gamma}}{d}\bigg)
		       \nonumber \\
\ll &\,\, \mathscr{L}\sum_{h\sim H_2}\sum_{m\sim M}\Bigg|\sum_{\ell\sim L}
		  e\bigg(\frac{hm^\gamma(\ell+u/m)^{\gamma}}{d}\bigg)\Bigg|.
\end{align*}
It follows from Lemma \ref{Robert-Sargos-lemma}, applied with parameters $(N,H,M)=(H_2,M,L)$, that
\begin{equation*}
\mathcal{K}_d\ll(MH_2L)^{1+\eta}\bigg(\bigg(\frac{Y}{MH_2L^2}\bigg)^{1/4}+\frac{1}{L^{1/2}}+\frac{1}{Y}\bigg),
\end{equation*}
where $Y=d^{-1}H_2M^{\gamma}L^{\gamma}$. Since $d\sim D\ll H_2$, we have $Y^{-1}\ll L^{-1/2}$. If
$M\ll X^{\gamma(1-\xi)-\eta}$, then
\begin{equation*}
\frac{1}{L^{1/2}}\ll\bigg(\frac{Y}{MH_2L^2}\bigg)^{1/4}.
\end{equation*}
Accordingly, under the condition $M\ll X^{(1-\xi)\gamma-\eta}$, there holds
\begin{align*}
            \mathcal{K}_d
 \ll & \,\, X^\eta(MH_2L)\bigg(\frac{Y}{MH_2L^2}\bigg)^{1/4}\ll X^\eta M^{(3+\gamma)/4}L^{(2+\gamma)/4}H_2d^{-1/4}
		           \nonumber \\
 \ll & \,\, X^{\gamma/4+1/2+\eta}M^{1/4}H_2D^{-1/4}\ll X^{(3\gamma\xi-3\gamma)/4+3/2+\eta}M^{1/4}\ll X^{1-\eta},
\end{align*}
provided that $M\ll X^{-2-3\gamma\xi+3\gamma-\eta}$. By noting that $(1-\xi)\gamma>-2-3\gamma\xi+3\gamma$ holds for $1/2<\gamma<1$, we obtain the following lemma.
\begin{lemma}\label{Type-I-es}
Suppose that M satisfies the condition
\begin{equation*}
M\ll X^{-2-3\gamma\xi+3\gamma-\eta}.
\end{equation*}
Then we have
\begin{equation*}
\Sigma_{I}\ll X^{1-\eta}.
\end{equation*}
\end{lemma}
Next, we complete the proof of the mean value estimate (\ref{mean-value-equilavent}). To handle the bilinear expressions on the right--hand side of (\ref{expo-fenjie}), we use the following decomposition lemma.
\begin{lemma}\label{exponen-fenjie}
If we have real numbers $0<\mathfrak{a}<1,\,0<\mathfrak{b}<\mathfrak{c}<1$  satisfying
\begin{equation*}
\mathfrak{b}<\frac{2}{3},\qquad 1-\mathfrak{c}<\mathfrak{c}-\mathfrak{b}, \qquad 1-\mathfrak{a}<\frac{\mathfrak{c}}{2},
\end{equation*}
then (\ref{expo-fenjie}) still holds when (\ref{gene-coeff-condi}) is replaced by the conditions
\begin{equation*}
M\leqslant X^\mathfrak{a} \qquad \textrm{for Type I sums},
\end{equation*}
and
\begin{equation*}
X^\mathfrak{b}\leqslant M\leqslant X^\mathfrak{c} \qquad \textrm{for Type II sums}.
\end{equation*}
\end{lemma}
\begin{proof}
See Proposition 1 of Balog and Friedlander \cite{Balog-Friedlander-1992}.  $\hfill$
\end{proof}
Combining Lemma \ref{Type-II-es}, Lemma \ref{Type-I-es} and Lemma \ref{exponen-fenjie}, we take
\begin{align*}
 \mathfrak{a}= & \,\, -2-3\gamma\xi+3\gamma-\eta,
		              \nonumber \\
 \mathfrak{b}= & \,\, \frac{8619(1-\gamma)}{1408}+\frac{48\gamma\xi}{11}+\eta,
		              \nonumber \\
 \mathfrak{c}= & \,\, \frac{33\gamma+192\gamma\xi+96}{173}-\eta.
\end{align*}
Thus, we are able to choose
\begin{equation}\label{level-def}
\xi=\xi(\gamma)=\min \bigg\{\frac{25857\gamma-23041}{18432\gamma}-\eta,\,\frac{357\gamma-314}{282\gamma}
-\eta\bigg\}
\end{equation}
For $0.98<\gamma<1$, it is easy to check the conditions  (\ref{suffi-condi-1}), (\ref{suffi-condi-2}), as well as the inequalities in Lemma \ref{exponen-fenjie}, hold. Accordingly, (\ref{mean-value-equilavent}) holds for $\xi$ defined as in (\ref{level-def}).

	\section{Proof of Theorem \ref{Theorem-1}}
	In this section, we shall prove Theorem \ref{Theorem-1} according to the weighted sieve of
	Richert \cite{Richert-1969} and the method of Chen \cite{Chen-1973}. Recall that
	\begin{equation*}
		\mathscr{A}=\big\{n: n\leqslant x, \, [n^{1/\gamma}] \textrm{ is a prime} \big\}.
	\end{equation*}
	We consider the weighted sum
	\begin{equation*}
		W\big(\mathscr{A},x^{1/26.18}\big)=\sum_{\substack{a\in\mathscr{A}\\ (a,P(x^{1/26.18}))=1}}
		\Bigg(1-\lambda\sum_{\substack{x^{1/26.18}\leqslant p<x^{1/u}\\ p|a}}\bigg(1-\frac{u\log p}{\log x}\bigg)\Bigg),
	\end{equation*}
	where $\lambda=(9-u-\varepsilon)^{-1}$, $u=\xi^{-1}+\varepsilon$, $\xi$ is defined as in (\ref{level-def}), and
	\begin{equation*}
		P(z)=\prod_{p<z} p.
	\end{equation*}
	For convenience, we write
	\begin{equation*}
		\mathscr{W}_a=1-\lambda\sum_{\substack{x^{1/26.18}\leqslant p<x^{1/u}\\ p|a}}\bigg(1-\frac{u\log p}{\log x}\bigg).
	\end{equation*}
	Then we have
	\begin{equation}\label{W-fenjie}
		W\big(\mathscr{A},x^{1/26.18}\big)
		=\sum_{\substack{a\in\mathscr{A}\\ (a,P(x^{1/26.18}))=1\\ \Omega(a)\leqslant7}}\mathscr{W}_a
		+\sum_{\substack{a\in\mathscr{A}\\ (a,P(x^{1/26.18}))=1\\ \Omega(a)=8\\ \mu(a)\not=0}}\mathscr{W}_a
		+\sum_{\substack{a\in\mathscr{A}\\ (a,P(x^{1/26.18}))=1\\ \Omega(a)\geqslant9\\ \mu(a)\not=0}}\mathscr{W}_a
		+\sum_{\substack{a\in\mathscr{A}\\ (a,P(x^{1/26.18}))=1\\ \Omega(a)\geqslant8\\ \mu(a)=0}}\mathscr{W}_a.
	\end{equation}
	Trivially, we have
	\begin{align}\label{n-s-free>8}
		\sum_{\substack{a\in\mathscr{A}\\ (a,P(x^{1/26.18}))=1\\ \Omega(a)\geqslant8\\ \mu(a)=0}}\mathscr{W}_a
		\ll  & \,\, \sum_{\substack{a\in\mathscr{A}\\ (a,P(x^{1/26.18}))=1\\ \mu(a)=0}}\tau(a)
		\ll x^\varepsilon \sum_{x^{1/26.18}\leqslant p_1\leqslant x^{1/2}}
		\sum_{\substack{a\leqslant x \\ a \equiv 0 \!\!\!\! \pmod{p_1^2} \\ p=[a^{1/\gamma}]}}1
		\nonumber \\
		\ll  & \,\, x^\varepsilon\sum_{x^{1/26.18}\leqslant p_1\leqslant x^{1/2}}\bigg(\frac{x}{p_1^2}+1\bigg)
		\ll  x^\varepsilon\big(x^{1-1/26.18}+x^{1/2}\big)\ll x^{\frac{25.18}{26.18}+\varepsilon}.
	\end{align}
	For given integer $a$ with $a\leqslant x$, $(a,P(x^{1/26.18}))=1$ and $\mu(a)\not=0$, the weight $\mathscr{W}_a$ in the sum $W(\mathscr{A},x^{1/26.18})$ satisfies
	\begin{align}\label{weight-upper}
		1-\lambda\sum_{\substack{x^{1/26.18}\leqslant p<x^{1/u}\\ p|a}}\bigg(1-\frac{u\log p}{\log x}\bigg)
		\leqslant & \,\,\lambda\bigg(\frac{1}{\lambda}-\sum_{p|a}\bigg(1-\frac{u\log p}{\log x}\bigg)\bigg)
		\nonumber \\
		= & \,\, \lambda\bigg(9-u-\varepsilon-\Omega(a)+\frac{u\log a}{\log x}\bigg)<\lambda(9-\Omega(a)),
	\end{align}
	and thus $\mathscr{W}_a<0$ for $\Omega(a)\geqslant9$. From (\ref{W-fenjie})--(\ref{weight-upper}), we know that
	\begin{align}\label{omega(a)<7-lower}
		\sum_{\substack{a\in\mathscr{A}\\ (a,P(x^{1/26.18}))=1\\ \Omega(a)\leqslant7}}\mathscr{W}_a
		= &\,\, W\big(\mathscr{A},x^{1/26.18}\big)-\sum_{\substack{a\in\mathscr{A}\\ (a,P(x^{1/26.18}))=1\\ \Omega(a)=8\\
				\mu(a)\not=0}}\mathscr{W}_a-\sum_{\substack{a\in\mathscr{A}\\ (a,P(x^{1/26.18}))=1\\ \Omega(a)\geqslant9\\ \mu(a)\not=0}}\mathscr{W}_a+O\big(x^{\frac{25.18}{26.18}+\varepsilon}\big)
		\nonumber \\
		\geqslant &\,\, W\big(\mathscr{A},x^{1/26.18}\big)-\sum_{\substack{a\in\mathscr{A}\\ (a,P(x^{1/26.18}))=1\\
				\Omega(a)=8\\ \mu(a)\not=0}}\mathscr{W}_a+O\big(x^{\frac{25.18}{26.18}+\varepsilon}\big).
	\end{align}
	Therefore, if we can show that the contribution of the second term on the right--hand side of (\ref{omega(a)<7-lower}) is strictly less than $(1-\varepsilon)W(\mathscr{A},x^{1/26.18})$, then we shall prove Theorem \ref{Theorem-1}. For $W(\mathscr{A},x^{1/26.18})$, we have
	\begin{align}\label{W(a)-chai}
		W\big(\mathscr{A},x^{1/26.18}\big)
		= & \,\, \sum_{\substack{a\in\mathscr{A}\\(a,P(x^{1/26.18}))=1}}1-\lambda\sum_{x^{1/26.18}\leqslant p<x^{1/u}}
		\bigg(1-\frac{u\log p}{\log x}\bigg)\sum_{\substack{a\in\mathscr{A}\\(a,P(x^{1/26.18}))=1\\ p|a}}1
		\nonumber\\
		= & \,\, S\big(\mathscr{A},x^{1/26.18}\big)-\lambda\sum_{x^{1/26.18}\leqslant p<x^{1/u}}
		\bigg(1-\frac{u\log p}{\log x}\bigg)S\big(\mathscr{A}_p,x^{1/26.18}\big).
	\end{align}
Now, we shall use Lemma \ref{upper-lower-sieve} to give the lower bound of $S\big(\mathscr{A},x^{1/26.18}\big)$. In view of (\ref{A_d-asymp}), we take
\begin{equation*}
\mathscr{X}=\sum_{p\leqslant x^{1/\gamma}}\big((p+1)^\gamma-p^\gamma\big),\qquad
 \omega(d)=
\begin{cases}
1, & \text{if } \mu(d)\not=0 \text{ and } d|P(x^{1/26.18}), \\
0, & \text{otherwise}.
\end{cases}
\end{equation*}
By (\ref{mean-value-equilavent}) and the arguments in Section \ref{ex-section}, we know that
\begin{equation*}
\sum_{d\leqslant x^\xi}\bigg|\#\mathscr{A}_d-\frac{\omega(d)}{d}\mathscr{X}\bigg|\ll\frac{x}{(\log x)^A},
\end{equation*}
where $\xi$ is defined by (\ref{level-def}). It follows from (\ref{diff-eq}) that
\begin{equation*}
F(s)=\frac{2e^{C_0}}{s},\qquad 0<s\leqslant3;\qquad f(s)=\frac{2e^{C_0}\log(s-1)}{s},
\qquad 2\leqslant s\leqslant4,
\end{equation*}
where $C_0$ denotes Euler's constant. By noting the fact that $2<26.18\xi<4$ holds for $0.942415<\gamma<1$, then
Lemma \ref{upper-lower-sieve} yields
\begin{align}
		         S\big(\mathscr{A},x^{1/26.18}\big)
\geqslant & \,\, \mathscr{X}V\big(x^{1/26.18}\big)\Big(f\big(26.18\xi\big)-o(1)\Big)
		               \nonumber \\
	    = & \,\, 2e^{C_0}\mathscr{X}V\big(x^{1/26.18}\big)
                 \Bigg(\frac{\log\big(26.18\xi-1\big)}{26.18\xi}-o(1)\Bigg),
\end{align}
where $C_0$ denotes Euler's constant. Moreover, it follows from (1.11) on p. 245  and (1.13) on p. 246 of Halberstam and Richert \cite{Halberstam-Richert-book} that
\begin{align}\label{A_p-W-lower}
		  & \,\, \sum_{x^{1/26.18}\leqslant p<x^{1/u}}\bigg(1-\frac{u\log p}{\log x}\bigg)
                 S\big(\mathscr{A}_p,x^{1/26.18}\big)
		               \nonumber \\
\leqslant & \,\, \mathscr{X}V\big(x^{1/26.18}\big)\Bigg(\sum_{x^{1/26.18}\leqslant p<x^{1/u}}
		         \bigg(1-\frac{u\log p}{\log x}\bigg)\frac{1}{p}\cdot
		         F\bigg(\frac{\log(x^\xi/p)}{\log x^{1/26.18}}\bigg)+o(1)\Bigg)
		               \nonumber \\
	    = & \,\, 2e^{C_0}\mathscr{X}V\big(x^{1/26.18}\big)
		         \bigg(\int_u^{26.18}\frac{t-u}{26.18t(\xi t -1)}\mathrm{d}t+o(1)\bigg).
\end{align}
Combining (\ref{W(a)-chai})--(\ref{A_p-W-lower}), we obtain
\begin{equation}\label{W(a,1/26.18)-lower}
W\big(\mathscr{A},x^{1/26.18}\big)\geqslant2e^{C_0}\mathscr{X}V\big(x^{1/26.18}\big)
\Bigg(\frac{\log\big(26.18\xi-1\big)}{26.18\xi}-\lambda\int_u^{26.18}\frac{t-u}{26.18t(\xi t-1)}
\mathrm{d}t+o(1)\Bigg).
\end{equation}
Now, we consider the second term on the right--hand side of (\ref{omega(a)<7-lower}). Set
\begin{equation*}
\mathscr{B}=\Big\{m:m\leqslant x,\, m=p_1p_2\dots p_8,\, x^{1/26.18}\leqslant p_1<p_2<\dots<p_8\Big\},
\end{equation*}
and
\begin{equation*}
\mathscr{E}=\Big\{n:n=[m^{1/\gamma}],\, m\in\mathscr{B}\Big\}.
\end{equation*}
From (\ref{weight-upper}) we deduce that
\begin{align}\label{E-trans-upper}
\sum_{\substack{a\in\mathscr{A}\\ (a,P(x^{1/26.18}))=1\\ \Omega(a)=8\\ \mu(a)\not=0}}\mathscr{W}_a
< & \,\, \lambda \sum_{\substack{a\in\mathscr{A}\\ (a,P(x^{1/26.18}))=1\\ \Omega(a)=8\\ \mu(a)\not=0}}1
=\lambda\sum_{\substack{a \leqslant x,\, p=[a^{1/\gamma}] \\ a=p_1p_2\dots p_8 \\
x^{1/26.18}\leqslant p_1<p_2<\dots<p_8}} 1
\leqslant\lambda\cdot S\big(\mathscr{E},x^{1/(2\gamma)}\big).
\end{align}
Let $\mathscr{E}_d=\big\{n\in\mathscr{E}: n\equiv0\!\pmod d\big\}$. Then one has $\#\mathscr{E}_d=\#\big\{\ell: \ell d=[m^{1/\gamma}],\,\,m\in\mathscr{B}\big\}$. Trivially, we have
\begin{equation*}
m^{1/\gamma}-1<\ell d \leqslant m^{1/\gamma}\quad\Longleftrightarrow \quad\frac{m^{1/\gamma}-1}{d}<\ell
\leqslant \frac{m^{1/\gamma}}{d},
\end{equation*}
and thus
\begin{align*}
		 \#\mathscr{E}_d
= & \,\, \sum_{m\in\mathscr{B}}\bigg(\bigg[\frac{m^{1/\gamma}}{d}\bigg]-\bigg[\frac{m^{1/\gamma}-1}{d}\bigg]\bigg)
		           \nonumber \\
= & \,\, \frac{1}{d}\sum_{m\in\mathscr{B}}1+\sum_{m\in\mathscr{B}}\bigg(\psi\bigg(\frac{m^{1/\gamma}-1}{d}\bigg)-
		 \psi\bigg(\frac{m^{1/\gamma}}{d}\bigg)\bigg)
		           \nonumber \\
=: & \,\, \frac{1}{d}\mathcal{X}+\mathscr{R}_d,
\end{align*}
where
\begin{equation}\label{8-main-error-def}
\mathcal{X}=\sum_{m\in\mathscr{B}}1,\qquad
\mathscr{R}_d=\sum_{m\in\mathscr{B}}\bigg(\psi\bigg(\frac{m^{1/\gamma}-1}{d}\bigg)\bigg)-
\psi\bigg(\frac{m^{1/\gamma}}{d}\bigg)\bigg).
\end{equation}
In order to use Lemma \ref{upper-lower-sieve} to give upper bound for $S(\mathscr{E},x^{1/(2\gamma)})$, we need to establish the following lemma.
\begin{lemma}\label{omega=8-error}
Let $\mathscr{R}_d$ be defined as in (\ref{8-main-error-def}). Then, for any $A>0$, we have
\begin{equation}
\sum_{d\leqslant x^\xi}\big|\mathscr{R}_d\big|\ll \frac{x}{(\log x)^{A}}.
\end{equation}
\end{lemma}
\begin{proof}
From the definition of $\mathscr{R}_d$ and splitting argument, it is sufficient to show that, for $X\leqslant x$,
there holds
\begin{equation}\label{error-splitting-1}
\sum_{d\sim D}\Bigg|\sum_{\substack{m\in\mathscr{B}\\ m\sim X}}
\bigg(\psi\bigg(\frac{m^{1/\gamma}-1}{d}\bigg)-\psi\bigg(\frac{m^{1/\gamma}}{d}\bigg)\bigg)\Bigg|
\ll\frac{x}{(\log x)^A},
\end{equation}
where $D=x^\xi$. Let $\mathds{1}_{\mathscr{B}}(\cdot)$ be the characteristic function supporting on $\mathscr{B}$. Then (\ref{error-splitting-1}) can be rewritten as
\begin{equation}\label{error-splitting-2}
\sum_{d\sim D}\Bigg|\sum_{ m\sim X}\mathds{1}_{\mathscr{B}}(m)
\bigg(\psi\bigg(\frac{m^{1/\gamma}-1}{d}\bigg)-\psi\bigg(\frac{m^{1/\gamma}}{d}\bigg)\bigg)\Bigg|
\ll\frac{x}{(\log x)^A}.
\end{equation}
If $X\leqslant x^{(1-\eta)}$, then the left--hand side of (\ref{error-splitting-2}) is
\begin{align}\label{8-error-f-1}
 = & \,\, \sum_{d\sim D}\Bigg|\sum_{m\sim X}\mathds{1}_{\mathscr{B}}(m)\bigg(-\frac{1}{d}
          +\bigg[\frac{m^{1/\gamma}}{d}\bigg]-\bigg[\frac{m^{1/\gamma}-1}{d}\bigg]\bigg)\Bigg|
			     \nonumber \\
\leqslant & \,\, \bigg(\sum_{d\sim D}\frac{X}{d}\bigg)+\sum_{d\sim D}\sum_{m\sim X}
			     \bigg(\bigg[\frac{m^{1/\gamma}}{d}\bigg]-\bigg[\frac{m^{1/\gamma}-1}{d}\bigg]\bigg)
			     \nonumber \\
= & \,\, 2\bigg(\sum_{d\sim D}\frac{X}{d}\bigg)+\sum_{d\sim D}\sum_{m\sim X}
			\bigg(\psi\bigg(\frac{m^{1/\gamma}-1}{d}\bigg)-\psi\bigg(\frac{m^{1/\gamma}}{d}\bigg)\bigg)
			     \nonumber \\
\ll & \,\, x^{1-\eta}+\sum_{d\sim D}\sum_{m\sim X}\bigg(\psi\bigg(\frac{m^{1/\gamma}-1}{d}\bigg)
           -\psi\bigg(\frac{m^{1/\gamma}}{d}\bigg)\bigg).
\end{align}
For the sum on the right--hand side of (\ref{8-error-f-1}), by Lemma \ref{psi-expansion} with $H=H_*:=x^{\xi-\eta}$ in (\ref{psi-expan}), we have
\begin{align*}
  & \,\, \sum_{d\sim D}\sum_{m\sim X}\bigg(\psi\bigg(\frac{m^{1/\gamma}-1}{d}\bigg)
         -\psi\bigg(\frac{m^{1/\gamma}}{d}\bigg)\bigg)
			       \nonumber \\
= & \,\, \sum_{d\sim D}\sum_{m\sim X}\Bigg(\sum_{0<|h|\leqslant H_*}\frac{1}{2\pi ih}
		 \bigg(e\bigg(\frac{hm^{1/\gamma}}{d}\bigg)-e\bigg(\frac{h(m^{1/\gamma}-1)}{d}\bigg)\bigg)
			       \nonumber \\
  & \,\, \hspace{5em} +O\bigg(g\bigg(\frac{m^{1/\gamma}}{d},H_*\bigg)\bigg)
		 +O\bigg(g\bigg(\frac{m^{1/\gamma}-1}{d},H_*\bigg)\bigg)\Bigg)
			       \nonumber \\
=: & \,\, \Theta_1+\Theta_2+\Theta_3,
\end{align*}
say. It suffices to deal with $\Theta_2$, since $\Theta_3$ can be treated exactly the same. By Lemma \ref{expo-pair-gernal} with $(\kappa,\ell)=(\frac{1}{2},\frac{1}{2})$, we have
\begin{align}\label{theta-2-deal}
			\Theta_2
 \ll & \,\, \sum_{d\sim D}\sum_{h=-\infty}^\infty|b_h|
			\Bigg|\sum_{m\sim X}e\bigg(\frac{hm^{1/\gamma}}{d}\bigg)\Bigg|
			      \nonumber \\
 \ll & \,\, \mathscr{L}XDH_*^{-1}+\sum_{d\sim D}\sum_{\substack{h=-\infty\\ h\not=0}}^\infty
			|b_h|\Bigg|\sum_{m\sim X}e\bigg(\frac{hm^{1/\gamma}}{d}\bigg)\Bigg|
			      \nonumber \\
 \ll & \,\, \mathscr{L}XDH_*^{-1}+\sum_{d\sim D}\sum_{\substack{h=-\infty\\ h\not=0}}^\infty
			|b_h|\bigg(\frac{d}{|h|}X^{1-{1/\gamma}}+\bigg(\frac{|h|}{d}\bigg)^{1/2}X^{1/(2\gamma)}\bigg)
			      \nonumber \\
 \ll & \,\, \mathscr{L}XDH_*^{-1}+\sum_{d\sim D}\Bigg(\sum_{0<|h|\leqslant H_*}\frac{1}{|h|}
			\bigg(\frac{d}{|h|}X^{1-{1/\gamma}}+\bigg(\frac{|h|}{d}\bigg)^{1/2}X^{1/(2\gamma)}\bigg)
			      \nonumber \\
	 & \hspace{9em} +\sum_{|h|>H_*}\frac{H_*}{h^2}
			\bigg(\frac{d}{|h|}X^{1-{1/\gamma}}+\bigg(\frac{|h|}{d}\bigg)^{1/2}X^{1/(2\gamma)}\bigg)\Bigg)
			      \nonumber \\
 \ll & \,\, \mathscr{L}XDH_*^{-1}+D^2X^{1-{1/\gamma}}+D^{1/2}H_*^{1/2}X^{1/(2\gamma)}
			      \nonumber \\
 \ll & \,\, \frac{x}{(\log x)^{A}},
\end{align}
provided that
\begin{equation*}
\xi + \frac{1}{2\gamma}<1.
\end{equation*}
For $\Theta_1$, by Lemma \ref{expo-pair-gernal} with $(\kappa,\ell)=(\frac{1}{2},\frac{1}{2})$, we get
\begin{align*}
			\Theta_1
   = & \,\, \sum_{d\sim D}\sum_{0<|h|\leqslant H_*}\frac{1-e(-h/d)}{2\pi ih}\sum_{m\sim X}
			e\bigg(\frac{hm^{1/\gamma}}{d}\bigg)
			       \nonumber \\
 \ll & \,\, \sum_{d\sim D}\sum_{0<|h|\leqslant H_*}\frac{1}{d}
			\bigg(\frac{d}{|h|}X^{1-{1/\gamma}}+\bigg(\frac{|h|}{d}\bigg)^{1/2}X^{1/(2\gamma)}\bigg)
			       \nonumber \\
 \ll & \,\, D\mathscr{L}X^{1-1/\gamma}+H_*^{3/2}X^{1/(2\gamma)}
			       \nonumber \\
 \ll & \,\, x^{1+\xi-1/\gamma}+x^{3\xi/2+1/(2\gamma)}
			       \nonumber \\
 \ll & \,\, x^{1-\eta},
\end{align*}
provided that
\begin{equation*}
\frac{3\xi}{2}+\frac{1}{2\gamma}<1.
\end{equation*}
Now, we assume that $x^{1-\eta}<X\leqslant x$, by (\ref{psi-expan}) we know that the total contribution of the error term in (\ref{psi-expan}) to the left--hand side of (\ref{error-splitting-2}) is
\begin{equation*}
\ll \sum_{d\sim D}\sum_{m\sim X}\bigg(g\bigg(\frac{m^{1/\gamma}-1}{d},H\bigg)+
g\bigg(\frac{m^{1/\gamma}}{d},H\bigg)\bigg)=:\mathfrak{S}_1^*+\mathfrak{S}_2^*.
\end{equation*}
One can show that $\mathfrak{S}_1^*,\mathfrak{S}_2^* \ll x\mathscr{L}^{-A}$ by following the same processes as those in (\ref{theta-2-deal}) under the conditions
\begin{equation*}
H=x^{\xi-\eta},\qquad \xi+\frac{1}{2\gamma}<1.
\end{equation*}
The contribution of the main term in (\ref{psi-expan}) to the left--hand side of (\ref{error-splitting-2}) is
\begin{align*}
 = & \,\, \sum_{d\sim D}\Bigg|\sum_{m\sim X}\mathds{1}_{\mathscr{B}}(m)\sum_{0<|h|\leqslant H}\frac{1}{2\pi ih}
		  \bigg(e\bigg(\frac{hm^{1/\gamma}}{d}\bigg)-e\bigg(\frac{h(m^{1/\gamma}-1)}{d}\bigg)\bigg)\Bigg|
			     \nonumber \\
\ll & \,\,\sum_{d\sim D}\sum_{0<h\leqslant H}\frac{1}{h}\Bigg|\sum_{m\sim X}\mathds{1}_{\mathscr{B}}(m)
		  \bigg(e\bigg(\frac{hm^{1/\gamma}}{d}\bigg)-e\bigg(\frac{h(m^{1/\gamma}-1)}{d}\bigg)\bigg)\Bigg|
		  =:\mathfrak{S},
\end{align*}
say. Consequently, it suffices to show that $\mathfrak{S}\ll x\mathscr{L}^{-A}$. It follows that
\begin{align*}
		 \mathfrak{S}
= & \,\, \sum_{d\sim D}\sum_{0<h\leqslant H}\frac{|1-e(-h/d)|}{h}\Bigg|\sum_{m\sim X}\mathds{1}_{\mathscr{B}}(m)
		 e\bigg(\frac{hm^{1/\gamma}}{d}\bigg) \Bigg|
			    \nonumber \\
\ll & \,\, \sum_{d\sim D}\frac{1}{d}\sum_{0<h\leqslant H}\Bigg|\sum_{m\sim X}
		   \mathds{1}_{\mathscr{B}}(m)e\bigg(\frac{hm^{1/\gamma}}{d}\bigg)\Bigg|,
\end{align*}
Therefore, we obtain
\begin{align}\label{error-splitting-3}
	 & \,\,  \sum_{d\sim D}\Bigg|\sum_{m\sim X}\mathds{1}_{\mathscr{B}}(m)
			 \bigg(\psi\bigg(\frac{m^{1/\gamma}-1}{d}\bigg)-\psi\bigg(\frac{m^{1/\gamma}}{d}\bigg)\bigg)\Bigg|
			        \nonumber \\
 \ll & \,\, x\mathscr{L}^{-A}+\max_{\substack{x^{1-\eta}<X\leqslant x}}
            \sum_{d\sim D}\frac{1}{d}\sum_{0<h\leqslant H}\Bigg|\sum_{m\sim X}
			\mathds{1}_{\mathscr{B}}(m)e\bigg(\frac{hm^{1/\gamma}}{d}\bigg)\Bigg|
			        \nonumber \\
 \ll & \,\, x\mathscr{L}^{-A}+\max_{\substack{x^{1-\eta}<X\leqslant x}}
			\sum_{d\sim D}\frac{1}{d}\sum_{0<h\leqslant H}\delta^*(d,h)\sum_{m\sim X}
			\mathds{1}_{\mathscr{B}}(m)e\bigg(\frac{hm^{1/\gamma}}{d}\bigg)
			        \nonumber \\
 \ll & \,\, x\mathscr{L}^{-A}+\max_{\substack{x^{1-\eta}<X\leqslant x}}
			\Bigg|\sum_{\substack{m\in\mathscr{B}\\ m\sim X}}\sum_{0<h\leqslant H}\sum_{d\sim D}
			\frac{\delta^*(d,h)}{d}e\bigg(\frac{hm^{1/\gamma}}{d}\bigg)\Bigg|,
\end{align}
where
\begin{equation*}
\big|\delta^*(d,h)\big|=1.
\end{equation*}
Denote by $\mathfrak{S}_0$ the triple summation in the absolute value on the right--hand side in (\ref{error-splitting-3}). In order to prove (\ref{error-splitting-2}), it is sufficient to show that
$\mathfrak{S}_0\ll X^{1-\eta}$. By dividing $m$ into two variables and following the same processes, except for replacing $\gamma$ with $1/\gamma$ and replacing the exponent pair $(\frac{11}{32},\frac{13}{24})$ with $(\frac{1}{2},\frac{1}{2})$, as demonstrated in Type II sum estimate, we deduce that (using the same notation as in (\ref{Sigma-2-fi-1}))
\begin{align}\label{S_0-upper-1}
	       |\mathfrak{S}_0|^2
\ll & \,\, X^{\eta}\big(X^{2-1/\gamma}TJD+XMTJD^{-1}+X^{1/(2\gamma)+2}M^{-1}J^{5/2}D^{-1/2}T^{-1/2}
			       \nonumber \\
& \,\, \qquad+X^{1/(2\gamma)}J^{3/2}D^{-3/2}T^{1/2}\big).
\end{align}
Now, we choose
\begin{equation}\label{T-chosen-S_0}
T=\big[X^{(1+2\gamma)/(3\gamma)}M^{-4/3}D^{4/3}\big].
\end{equation}
Trivially, one has $T>1$ provided that $M\ll X^{1/(4\gamma)+\xi+1/2-\eta}$. Inserting (\ref{T-chosen-S_0}) into (\ref{S_0-upper-1}),  we obtain
\begin{align}\label{S_0-upper-2}
		   |\mathfrak{S}_0|^2
\ll & \,\, X^{\eta}\big(M^{-4/3}X^{(-2+8\gamma)/(3\gamma)}JD^{7/3}+M^{-1/3}X^{(1+5\gamma)/(3\gamma)}JD^{1/3}
			       \nonumber \\
    & \,\, \qquad +M^{-2/3}X^{(2+\gamma)/(3\gamma)}J^{3/2}D^{-5/6}\big),
\end{align}
which combined with $J\ll H=X^{\xi-\eta}$ and $X^{\xi}\leqslant D\leqslant X^{\xi+\eta/2}$ yields that
\begin{align*}
			|\mathfrak{S}_0|^2
 \ll & \,\, X^{\eta}\big(M^{-4/3}X^{(-2+8\gamma+10\gamma\xi)/(3\gamma)}
            +M^{-1/3}X^{(1+5\gamma+4\gamma\xi)/(3\gamma)}+M^{-2/3}X^{(2+\gamma+2\gamma\xi)/(3\gamma)}\big).
\end{align*}
Accordingly, one derives that $\mathfrak{S}_0\ll X^{1-\eta}$, provided that
\begin{equation}\label{S_0-upp-condition}
X^{1/\gamma-1+4\xi+\eta}\ll M\ll X^{1/(4\gamma)+\xi+1/2-\eta},
\end{equation}
and $2\xi + 1/(2\gamma) < 1$. By (\ref{level-def}), we get $\mathfrak{S}_0\ll X^{1-\eta}$ under the condition (\ref{S_0-upp-condition}). In particular, for $0.98352<\gamma<1$ and the definition of $\xi$, i.e., (\ref{level-def}), it is easy to see that
\begin{equation*}
X^{1/\gamma-1+4\xi+\eta}\leqslant X^{\alpha_0}<X^{\beta_0}\leqslant X^{1/(4\gamma)+\xi+1/2-\eta},
\end{equation*}
where
\begin{equation*}
\alpha_0=0.609964,\qquad \beta_0=0.886021.
\end{equation*}
Next, we shall illustrate that, for $m=p_1p_2\dots p_8\in\mathscr{B}$ with $m\sim X>x^{1-\eta}$, there must be some partial product of $p_1p_2\dots p_8$ which lies in the interval $[X^{\alpha_0},X^{\beta_0}]$.
		
First, since $p_i\geqslant x^{1/26.18}$ and $p_1p_2\dots p_8\in[x^{1-\eta},x]$, we have
$p_i\leqslant X^{\beta_0}\,(i=1,2,\dots,8)$. Otherwise, if there exists some $p_i>X^{\beta_0}$, then
$m\geqslant X^{\beta_0}\cdot x^{7/26.18}>x^{0.886021+7/26.18}>x^{1.153}$, which is a contradiction. If there exists
some $p_i\in[X^{\alpha_0},X^{\beta_0}]$, then the conclusion follows. If this case does not exist, we consider the
product $p_5p_6p_7p_8$. We claim that there must be $p_5p_6p_7p_8\leqslant X^{\beta_0}$. Otherwise, we get
\begin{equation*}
X^{\beta_0}<p_5p_6p_7p_8<X(p_1p_2p_3p_4)^{-1}<x^{1-4/26.18}<x^{0.8473}<X^{\beta_0},
\end{equation*}
which is a contradiction. If $p_5p_6p_7p_8\in[X^{\alpha_0},X^{\beta_0}]$, then the conclusion follows. If this case does not exist, i.e., $p_5p_6p_7p_8<X^{\alpha_0}$, we consider the product $p_3p_4p_5p_6p_7p_8$. Now, we claim that there must hold $p_3p_4p_5p_6p_7p_8\geqslant X^{\alpha_0}$. Otherwise, we obtain
\begin{equation*}
X\ll(p_1p_2)(p_3p_4p_5p_6p_7p_8)<(p_3p_4p_5p_6p_7p_8)^{4/3}<X^{4\alpha_0/3}<X^{0.82},
\end{equation*}
which is a contradiction. If $p_3p_4p_5p_6p_7p_8\leqslant X^{\beta_0}$, then the conclusion follows. If this case does not exist, i.e., $p_3p_4p_5p_6p_7p_8> X^{\beta_0}$, we have (under the condition $p_5p_6p_7p_8<X^{\alpha_0}$)
\begin{equation}\label{1234-condition}
X^{1-\alpha_0-\eta}<(X/2)(p_5p_6p_7p_8)^{-1}<p_1p_2p_3p_4<(p_1p_2p_3p_4p_5p_6p_7p_8)^{1/2}<X^{1/2},
\end{equation}
which implies that $p_3p_4>X^{(1-\alpha_0)/2-\eta}$. By noting that
\begin{equation*}
p_3p_4<p_5p_6<(p_5p_6p_7p_8)^{1/2}<X^{\alpha_0/2},
\end{equation*}
we get $X^{(1-\alpha_0)/2-\eta}<p_3p_4<X^{\alpha_0/2}$. If $X^{1-\beta_0}\leqslant p_3p_4<X^{\alpha_0/2}$, then
we have
\begin{equation*}
X^{\alpha_0}<X^{1-\alpha_0/2-\eta}<p_1p_2p_5p_6p_7p_8\leqslant X^{\beta_0}.
\end{equation*}
If $X^{(1-\alpha_0)/2-\eta}<p_3p_4<X^{1-\beta_0}$, we have $p_3<(p_3p_4)^{1/2}<X^{(1-\beta_0)/2}$, and thus
\begin{equation*}
p_4p_5p_6p_7p_8>X^{\beta_0}p_3^{-1}>X^{\beta_0-(1-\beta_0)/2}>X^{\alpha_0}.
\end{equation*}
Moreover, there must hold $p_4p_5p_6p_7p_8\leqslant X^{\beta_0}$. Otherwise, if $p_4p_5p_6p_7p_8> X^{\beta_0}$, then
\begin{equation}\label{123-condition}
p_1p_2p_3<X(p_4p_5p_6p_7p_8)^{-1}<X^{1-\beta_0}.
\end{equation}
From (\ref{1234-condition}) and (\ref{123-condition}), we deduce that
\begin{equation*}
X^{\beta_0-\alpha_0-\eta}<X^{1-\alpha_0-\eta}(p_3p_4)^{-1}<p_1p_2<(p_1p_2p_3)^{2/3}<X^{2(1-\beta_0)/3},
\end{equation*}
which is in contradiction to the fact that
\begin{equation*}
\beta_0-\alpha_0-\eta>\frac{2}{3}(1-\beta_0).
\end{equation*}
This completes the proof of Lemma \ref{omega=8-error}.
\end{proof}
From Lemma \ref{upper-lower-sieve} and Lemma \ref{omega=8-error}, we deduce that
\begin{equation}\label{S(E)-upper}
S\big(\mathscr{E},x^{1/(2\gamma)}\big)\leqslant \mathcal{X}V\big(x^{1/(2\gamma)}\big)\bigg(F\big(2\gamma\xi\big)+o(1)\bigg).
\end{equation}
According to (\ref{V(z)-def}), we get
\begin{equation*}
V\big(x^{1/(2\gamma)}\big)=\frac{2\gamma}{26.18}V\big(x^{1/26.18}\big)\big(1+O(\log x)^{-1}\big),
\end{equation*}
which combined with (\ref{S(E)-upper}) yields
\begin{equation}\label{S(E)-upper-1}
S(\mathscr{E},x^{1/(2\gamma)})\leqslant \frac{2e^{C_0}}{26.18\xi}\mathcal{X}V\big(x^{1/26.18}\big)(1+o(1)).
\end{equation}
Next, we compute the quantity $\mathcal{X}$ definitely. Obviously, we have
\begin{align}\label{X-first-num}
		  \mathcal{X}=&\sum_{m\in\mathscr{B}}1
		          \nonumber \\
 = & \,\, \sum_{x^{1/26.18}\leqslant p_1< x^{1/8}}\sum_{p_1<p_2<(x/p_1)^{1/7}}
		  \sum_{p_2<p_3<(x/(p_1p_2))^{1/6}}\sum_{p_3<p_4<(x/p_1p_2p_3)^{1/5}}
		  \sum_{p_4<p_5<(x/p_1p_2p_3p_4)^{1/4}}
		          \nonumber \\
   & \,\, \times\sum_{p_5<p_6<(x/p_1p_2p_3p_4p_5)^{1/3}}\sum_{p_6<p_7<(x/p_1p_2p_3p_4p_5p_6)^{1/2}}
		  \sum_{p_7<p_8<(x/p_1p_2p_3p_4p_5p_6p_7)}1
		          \nonumber \\
 = & \,\, \int_{x^{1/26.18}}^{x^{1/8}}\int_{u_1}^{(\frac{x}{u_1})^{1/7}}
		  \int_{u_2}^{(\frac{x}{u_1u_2})^{1/6}}\int_{u_3}^{(\frac{x}{u_1u_2u_3})^{1/5}}
		  \int_{u_4}^{(\frac{x}{u_1u_2u_3u_4})^{1/4}}\int_{u_5}^{(\frac{x}{u_1u_2u_3u_4u_5})^{1/3}}
		          \nonumber \\
   & \,\, \times\int_{u_6}^{(\frac{x}{u_1u_2u_3u_4u_5u_6})^{1/2}}
		  \int_{u_7}^{\frac{x}{u_1u_2u_3u_4u_5u_6u_7}}
		  \frac{\mathrm{d}u_8\mathrm{d}u_7\dots\mathrm{d}u_2\mathrm{d}u_1}{(\log u_1)(\log u_2)\dots(\log u_8)}
		          \nonumber \\
 = & \,\, \int_{\frac{1}{26.18}}^{\frac{1}{8}}\frac{\mathrm{d}t_1}{t_1}
		  \int_{t_1}^{\frac{1-t_1}{7}}\frac{\mathrm{d}t_2}{t_2}
		  \int_{t_2}^{\frac{1-t_1-t_2}{6}}\frac{\mathrm{d}t_3}{t_3}
		  \int_{t_3}^{\frac{1-t_1-t_2-t_3}{5}}\frac{\mathrm{d}t_4}{t_4}
		  \int_{t_4}^{\frac{1-t_1-t_2-t_3-t_4}{4}}\frac{\mathrm{d}t_5}{t_5}
		          \nonumber \\
   & \,\, \times\int_{t_5}^{\frac{1-t_1-t_2-t_3-t_4-t_5}{3}}\frac{\mathrm{d}t_6}{t_6}
		  \int_{t_6}^{\frac{1-t_1-t_2-t_3-t_4-t_5-t_6}{2}}\frac{\mathrm{d}t_7}{t_7}
		  \int_{t_7}^{1-t_1-t_2-t_3-t_4-t_5-t_6-t_7}
		  \frac{x^{t_1+t_2+\dots+t_8}}{t_8}\mathrm{d}t_8.
\end{align}
For the innermost integral in (\ref{X-first-num}), we have
\begin{align}\label{inner-expli}
   & \,\, \int_{t_7}^{1-t_1-t_2-t_3-t_4-t_5-t_6-t_7}\frac{x^{t_1+t_2+\dots+t_8}}{t_8}\mathrm{d}t_8
		          \nonumber \\
 = & \,\, \frac{1}{\log x}\int_{t_7}^{1-t_1-t_2-t_3-t_4-t_5-t_6-t_7}\frac{1}{t_8}\mathrm{d}x^{t_1+t_2+\dots+t_8}
		          \nonumber \\
 = & \,\, \frac{1}{\log x}\bigg(\frac{x}{1-t_1-t_2-t_3-t_4-t_5-t_6-t_7}+O\bigg(\frac{x}{\log x}\bigg)\bigg)
		          \nonumber \\
 = & \,\, \frac{1}{1-t_1-t_2-t_3-t_4-t_5-t_6-t_7}\cdot\frac{x}{\log x}(1+o(1)).
\end{align}
From (\ref{X-first-num}), (\ref{inner-expli}), we deduce that
\begin{align}\label{X-compu-num}
		  \mathcal{X}
 = & \,\, \frac{x(1+o(1))}{\log x}\int_{\frac{1}{26.18}}^{\frac{1}{8}}\frac{\mathrm{d}t_1}{t_1}
		  \int_{t_1}^{\frac{1-t_1}{7}}\frac{\mathrm{d}t_2}{t_2}
		  \int_{t_2}^{\frac{1-t_1-t_2}{6}}\frac{\mathrm{d}t_3}{t_3}
		  \int_{t_3}^{\frac{1-t_1-t_2-t_3}{5}}\frac{\mathrm{d}t_4}{t_4}
		  \int_{t_4}^{\frac{1-t_1-t_2-t_3-t_4}{4}}\frac{\mathrm{d}t_5}{t_5}
		              \nonumber \\
   & \,\, \times\int_{t_5}^{\frac{1-t_1-t_2-t_3-t_4-t_5}{3}}\frac{\mathrm{d}t_6}{t_6}
		  \int_{t_6}^{\frac{1-t_1-t_2-t_3-t_4-t_5-t_6}{2}}
		  \frac{\mathrm{d}t_7}{t_7(1-t_1-t_2-t_3-t_4-t_5-t_6-t_7)}.
\end{align}
Combining (\ref{omega(a)<7-lower}), (\ref{W(a,1/26.18)-lower}), (\ref{E-trans-upper}), (\ref{S(E)-upper-1}) and (\ref{X-compu-num}), we obtain
\begin{align*}
		         \sum_{\substack{a\in\mathscr{A}\\(a,P(x^{1/26.18}))=1\\ \Omega(a)\leqslant7}}\mathscr{W}_a
\geqslant & \,\, \frac{2e^{C_0}}{26.18}\frac{\gamma x}{\log x}V\big(x^{1/26.18}\big)(1+o(1))
		         \Bigg(\frac{\log\big(26.18\xi-1\big)}{\xi}-\lambda\int_u^{26.18}\frac{t-u}{t(t\xi-1)}\mathrm{d}t
		               \nonumber \\
	& \,\, -\frac{\lambda}{\xi \gamma}\int_{\frac{1}{26.18}}^{\frac{1}{8}}\frac{\mathrm{d}t_1}{t_1}
		   \int_{t_1}^{\frac{1-t_1}{7}}\frac{\mathrm{d}t_2}{t_2}
		   \int_{t_2}^{\frac{1-t_1-t_2}{6}}\frac{\mathrm{d}t_3}{t_3}
		   \int_{t_3}^{\frac{1-t_1-t_2-t_3}{5}}\frac{\mathrm{d}t_4}{t_4}
		   \int_{t_4}^{\frac{1-t_1-t_2-t_3-t_4}{4}}\frac{\mathrm{d}t_5}{t_5}
		               \nonumber \\
	 & \,\, \times\int_{t_5}^{\frac{1-t_1-t_2-t_3-t_4-t_5}{3}}\frac{\mathrm{d}t_6}{t_6}
		    \int_{t_6}^{\frac{1-t_1-t_2-t_3-t_4-t_5-t_6}{2}}
		    \frac{\mathrm{d}t_7}{t_7(1-\sum_{j=1}^7t_j)}\Bigg)+O\big(x^{\frac{25.18}{26.18}+\varepsilon}\big).
\end{align*}
By direct numerical calculations, it is easy to see that the number in the above brackets $(\cdot)$ is $\geqslant0.00167598$, provided that $0.98353<\gamma<1$. This completes the proof of Theorem \ref{Theorem-1}.

\section*{Acknowledgement}
	
\noindent
The authors would like to appreciate the referee for his/her patience in refereeing this paper. This work is supported by Beijing Natural Science Foundation (Grant No. 1242003), and the National Natural Science Foundation
of China (Grant Nos. 12471009, 12301006, 11901566, 12001047).
	
\subsection*{Author Contributions}
All authors contribute equally to this work. The manuscript is approved by all authors for publication.

\subsection*{Code Availability}
 Not applicable.

\subsection*{Availability of Data and Materials}
Data sharing not applicable to this article as no data sets were generated or analysed during the current study.

\subsection*{Conflict of Interest}

The authors declare no conflicts of interest. All procedures were in accordance with the ethical standards
of the institutional research committee and with the 1964 Helsinki declaration and its later amendments or comparable ethical standards.

\end{document}